\documentclass[11pt, a4paper]{article} 

\usepackage{soul} 
\usepackage{color, xcolor}  
\usepackage[margin=1in]{geometry} 
\usepackage{amsfonts, amscd, amssymb, mathtools, mathrsfs, dsfont, bbm, bbding} 
\usepackage[amsmath, amsthm, thmmarks]{ntheorem} 
\usepackage{graphicx, xypic, color, float} 
\usepackage{indentfirst}
\usepackage{lmodern}
\usepackage[T1]{fontenc} 
\usepackage{enumerate, listings, verbatim, paralist}
\usepackage{setspace, xspace}
\usepackage[colorlinks=true, citecolor=blue]{hyperref}
\usepackage{extarrows}
\usepackage{pdfpages}
\usepackage{ulem}

\usepackage{setspace}
\AtBeginDocument{%
\addtolength\abovedisplayskip{-0.15\baselineskip}%
\addtolength\belowdisplayskip{-0.15\baselineskip}%
\addtolength\abovedisplayshortskip{-0.15\baselineskip}%
\addtolength\belowdisplayshortskip{-0.15\baselineskip}%
}

\newtheorem{thm}{Theorem} [section]

\newtheorem{lemma}[thm]{Lemma}
\newtheorem{defn}{Definition}[section]
\newtheorem{assmp}{Assumption}[section]

\numberwithin{equation}{section} 

\renewcommand{\geq}{\geqslant}
\renewcommand{\leq}{\leqslant}
\renewcommand{\ge}{\geqslant}

\newcommand{\citethm}[1]{Theorem \ref{#1}}

\newcommand{\citelem}[1]{Lemma \ref{#1}}

\newcommand{\opfont}{\mathbb}

\newcommand{\BE}[2][]{\ensuremath{\operatorname{\opfont{E}}^{#1}\!\left[#2\right]}}
\newcommand{\bp}{\ensuremath{\opfont{P}}}

\newcommand{\BF}{\ensuremath{\mathcal{F}}}

\newcommand{\R}{\ensuremath{\operatorname{\mathbb{R}}}}

\newcommand{\dd}{\ensuremath{\operatorname{d}\!}}
\newcommand{\dt}{\ensuremath{\operatorname{d}\! t}}
\newcommand{\ds}{\ensuremath{\operatorname{d}\! s}}

\newcommand{\dw}{\ensuremath{\operatorname{d}\! W}}

\newcommand{\idd}[1]{\ensuremath{\operatorname{\mathbf{1}}_{#1}}}

\newcommand{\barv}{\bar{v}}

\newcommand{\bart}{\bar{t}}

\newcommand{\barx}{\bar{x}}
\newcommand{\bary}{\bar{y}}
\newcommand{\barc}{\bar{c}}

\newcommand{\unc}{\underline{c}}
\newcommand{\widt}{\widetilde{t}}
\newcommand{\widx}{\widetilde{x}}
\newcommand{\widc}{\widetilde{c}}
\newcommand{\wids}{\widetilde{s}}
\newcommand{\widy}{\widetilde{y}}
\newcommand{\widd}{\widetilde{d}}

\newcommand{\nn}{\nonumber}

\newcommand{\setw}{\mathscr{W}}
\newcommand{\setg}{\mathscr{G}}
\newcommand{\setq}{\mathcal{Q}} 
\newcommand{\seta}{\mathcal{A}}

\newcommand{\setad}{\mathscr{C}}

\newcommand{\ep}{\varepsilon}

\newcommand{\sig}{\sigma} 
\newcommand{\setc}{\mathcal{C}} 
\newcommand{\rn}{\R^{N}}

\newcommand{\hatc}{\hat{c}}
\newcommand{\hatd}{\hat{d}}
\newcommand{\hatr}{\hat{r}}
\newcommand{\hats}{\hat{s}}
\newcommand{\hatt}{\hat{t}}
\newcommand{\hatp}{\hat{p}}
\newcommand{\hatq}{\hat{q}}
\newcommand{\hatu}{\hat{u}}
\newcommand{\hatv}{\hat{v}}
\newcommand{\hatx}{\hat{x}}
\newcommand{\hatX}{\widehat{X}}
\newcommand{\haty}{\hat{y}}

\newcommand{\trace}[1]{\operatorname{tr}\left(#1\right)}

\begin{document}
\title{Stochastic optimal self-path-dependent control: \\ A new type of variational inequality and its viscosity solution} 
\date{\today}
\author{Mingxin Guo\thanks{Department of Applied Mathematics, The Hong Kong Polytechnic University, Kowloon, Hong Kong, China.  Email: \url{mingxin.guo@connect.polyu.hk}
}
 \and Zuo Quan Xu\thanks{Department of Applied Mathematics, The Hong Kong Polytechnic University, Kowloon, Hong Kong, China.  Email: \url{maxu@polyu.edu.hk}
}}
\maketitle
\begin{abstract} 
In this paper, we explore a new class of stochastic control problems characterized by specific control constraints. Specifically, the admissible controls are subject to the ratcheting constraint, meaning they must be non-decreasing over time and are thus self-path-dependent. This type of problems is common in various practical applications, such as optimal consumption problems in financial engineering and optimal dividend payout problems in actuarial science.
Traditional stochastic control theory does not readily apply to these problems due to their unique self-path-dependent control feature. To tackle this challenge, we introduce a new class of Hamilton-Jacobi-Bellman (HJB) equations, which are variational inequalities concerning the derivative of a new spatial argument that represents the historical maximum control value. Under the standard Lipschitz continuity condition, we demonstrate that the value functions for these self-path-dependent control problems are the unique solutions to their corresponding HJB equations in the viscosity sense.
\end{abstract}
%

\section{Introduction}
Determining the payout strategy will not only influence the behaviors of shareholders, but will affect the financial health of a company as well, thus playing a fundamental role in risk management for companies. A central theme in this domain is the optimal dividend payout problem, with first contributions from \cite{de1957impostazione} and  \cite{gerber1969entscheidungskriterien}, where they found the optimal dividend payout policy with surplus process, establishing the foundations for further research. The problem has then been studied in many variants concerning the surplus process. For example,  \cite{gerber2006optimal} and \cite{albrecher2020optimal} considered the model with Possion process; \cite{asmussen1997controlled}, \cite{gerber2004optimal}, \cite{azcue2005optimal} dealt with   Brownian motion model; \cite{belhaj2010optimal} solved the jump-diffusion model and \cite{reppen2020optimal} settled the stochastic profitability model. Also there are some studies about variants concerning the constraint (for example \cite{KRS21} considered capital injections); \cite{avanzi2009strategies} provides an invaluable overview to the management strategies. 

In this paper, we diverge from traditional optimal payout problems, focusing instead on finding the optimal value function with self-path-dependent constraint where the dividend rate paid by company can never be reduced. In real life, dividend strategies are often designed to be non-decreasing over time to avoid sending negative signals to the market. Thus, we exclusively consider strategies where the rate remains above historical values, a condition known as the self-path-dependent constraint or ratcheting constraint. The conception of ratcheting constraint was first introduced in the context of lifetime consumption by \cite{Dybvig1995dusenberry}, where he built an infinite horizon lifetime portfolio selection model with consumption ratchet constraints, in which the agent's standard life does not allow for a decline. \cite{albrecher2018dividends}   solved a ratcheting constraint problem, who performed explicit calculations for a restricted ratcheting strategy, allowing a single increase in the dividend rate during the process's lifetime. This strategy was studied within both the Cramer-Lundberg model and its diffusion approximation, revealing surprising relationships between the optimal ratcheting level and the threshold level of unconstrained dividend strategies. Similar problems with wealth ratcheting have been explored by \cite{herve2006optimal} and \cite{romuald2008optimal}. Besides the foundational model proposed by \cite{Dybvig1995dusenberry}, the problem about the monotone of the rate has been studied by \cite{arun2012merton}, \cite{angoshtari2019optimal}, and \cite{deng2022optimal}, who allow the rate to decrease but no less than a fixed rate. Further exploration into similar optimal consumption problems has been conducted by \cite{jeon2018portfolio} and \cite{albrecher2022optimal}, who examined the horizon effect under ratcheting or path-dependent constraint within a finite time investment cycle. \cite{chen2015on} and \cite{deng2022optimal} considered the problem incorporating habit-formation, modeled as the exponentially weighted average of past consumption rates. For the dividend payout problem under both Brownian and compound Poisson surplus processes, \cite{albrecher2020optimal} and \cite{albrecher2022optimal} studied the optimal dividend problem with ratcheting constraints, defining admissible dividend policies as non-decreasing, right-continuous processes and they solve the problem in viscosity sense, which was proposed by \cite{azcue2014stochastic}. Recently, \cite{guan2024optimal1} and \cite{guan2023optimal2} studied the optimal dividend problem with ratcheting constraints using PDE arguments, successfully deriving an optimal feedback payout strategy. Also by using the techniques of viscosity solution, \cite{reppen2020optimal} investigated the optimal dividend problem with a stochastic profitability model.

From a technical point of view, the self-path-dependent problem is different from the traditional stochastic optimal problem. The related Hamilton-Jacobi-Bellman (HJB) equations for this new type of problems are variational inequalities with at least two arguments: a state argument (representing the surplus level) and a control argument (representing the historical maximum dividend payout rate). To the best of our knowledge, there seems no universal way to deal with such type of variational inequalities. Unlike classical variational inequalities with function constraints, such as the Black-Scholes PDE for American options, the new type of variational inequalities involve a gradient constraint on the value function with respect to (w.r.t.) the control argument.  Although similar HJB equations with gradient constraints appear in problems involving transaction costs (see \cite{dai2009finite}, and \cite{dai2010continuous}), they differ fundamentally. In transaction cost problems, gradient constraints apply to the state argument, whereas, in dividend ratcheting problems, they apply to the control argument, necessitating different methodologies. \cite{Dybvig1995dusenberry} also studied the ratcheting problem, where they used geometric Brownian motion as the underlying process and applied a logarithmic or power utility function to the consumption rate. However, they setup the model in the framework of Merton-type consumption problems. Although there are some apparent similarities, the current risk theory setup does not fall into the category of models examined in those studies, and the methods used for its analysis are quite distinct. Building on the extensive research in this area, we extend the problem to a general nonlinear setting. We eliminate the consideration of ruin time, and   fix the end time, leading to the analysis of a three-dimensional stochastic control problem, which is more intricate than those addressed in previous studies. 
Finding a strong solution to the new type of variational inequalities  is challenging. Instead of solving the PDE as in \cite{guan2023optimal2}, our resolution of this three-dimensional ratcheting problem culminates in establishing a viscosity solution and a comparison theorem, significantly contributing to the literature and expanding the scope of stochastic control theory.

The   remainder  of this paper  is organized as follows: Section \ref{PF} introduces the model and delves into the detailed intricacies of the problem. Section \ref{BR} provides foundational results on the properties of Lipschitz functions and presents the dynamic programming principle. Section \ref{VS} defines the HJB equation and corresponding viscosity solutions. We also prove the existence and uniqueness of HJB equation in viscosity sense.

\section{Problem Formulation}\label{PF}

We fix a probability space $(\Omega, \BF, \bp)$ throughout the paper. 
Let $W(\cdot)$ be a standard $m$-dimensional Brownian motion in the probability space and $\{\BF_t\}_{t\geq 0}$ be its natural filtration. Assume $\BF=\BF_T$, where $T$ is a fixed positive constant. 

We assume an $N$-dimensional controlled process $X(\cdot)$ follows the following stochastic differential equation:
\begin{align}\label{state1}
\dd X(s)=b(s,X(s),u(s))\dt+\sig(s,X(s),u(s))\dw_{s}, 
\end{align}
where $u(t)$ is the control at time $t$. 
The cost functional is given by 
\begin{align} \label{cost1}
J(t, x, u(\cdot)):=\BE[t,x]{\int_{t}^{T}f(s, X(s), u(s))\ds+H(X(T))},
~(t,x)\in [0,T]\times\rn,
\end{align}
where $\BE[t,x]{\:\cdot\:}$ means the conditional expectation given the initial condition $X(t)=x$. %
Our problem is to minimize the cost functional \eqref{cost1} over a special set of admissible controls, where the controlled process $X(\cdot)$ follows \eqref{state1}.

A distinctive characteristic of the control problem \eqref{cost1} we consider here is that the controls $u(\cdot)$ must be subject to the so-called ratcheting constraint, namely, they must be non-decreasing over time, rendering them inherently self-path-dependent as their future values are influenced by their historical values. 
This type of self-path-dependent control problems unfortunately falls outside the purview of classical control theory where the control constraint set at any time does not relay on their historical values (see, e.g., \cite{YZ99}), and has been rarely addressed in the existing literature. However, 
this type of controls are prevalent in numerous practical applications, such as in models for dividend payout strategies and consumption plans that incorporate drawdown constraints; see, e.g., \cite{albrecher2023optimal,guan2024optimal1,guan2023optimal2}.
But all these works assumed linear controlled processes and linear cost functionals. 
The aim of this paper is to establish a general result for a class of self-path-dependent control problems that satisfy the usual Lipchitz continuity conditions, thereby filling a gap in the field and providing a framework that can be applied to a variety of real-world scenarios where self-path-dependent controls play a critical role.

Mathematically speaking, for any two real constants $c_{1}<c_{2}$, not necessarily positive, 
we denoted by $\setad[c_{1}, c_{2}]$ the set of admissible controls $u(\cdot):[0,T]\to [c_{1}, c_{2}]$ which are $\{\BF_t\}_{0\leq t\leq T}$-adapted, right-continuous, and nondecreasing over time.  Throughout the paper, we fix a domain: $$\setq=[0,T]\times\rn\times(\unc,\barc],$$ 
and denote its closure by $$\overline\setq=[0,T]\times\rn\times [\unc,\barc],$$ 
where $\unc<\barc$ are two fixed real constants.  
Our controls are nondecreasing, so the value $\unc$ is not so important; 
we only need it to ensure our estimates hold uniformly. By contrast, the upper bound $\barc$ is more critical, as it makes the HJB equation associated with our problem is subject to a bounded domain w.r.t. the argument $c$. Practically speaking, any decision maker shall not be allowed to payout dividends at a higher rate than its income rate, thus the payout rate (controlled by the decision maker) is upper bounded by the income rate. Hence, putting an upper bound $\barc$ in our problem is reasonable. Without such an upper bound constraint, there are more technical issues to address, which is left to our future work.

Our goal is to solve the following stochastic optimal control problem: 
\begin{align}\label{p1}
V(t, x, c):=\inf_{u(\cdot)\in\setad[c,\barc]} J(t, x, u(\cdot)),~~ (t, x, c)\in \setq. 
\end{align}
Here $V$ is called the value function of the problem \eqref{p1}. If a control $u^{*}(\cdot)\in\setad[c,\barc]$ satisfies 
\begin{align}\label{p2}
J(t, x, u^{*}(\cdot))=V(t, x, c),
\end{align}
then it is called an optimal control to the problem \eqref{p1} for the initial position $(t,x,c)$. 

One can easily see from the definition \eqref{p1} that, as long as $c\in(\unc,\barc]$, the value function $V$ does not relay on the lower bound $\unc$ since the admissible set $\setad[c,\barc]$ does not. Moreover, the value function $V$ is non-decreasing (resp. decreasing) w.r.t. $c$ (resp. $\barc$) since the admissible set $\setad[c,\barc]$ is getter shrinking (resp. expanding) as $c$ (resp. $\barc$) increases.  

Regarding the control problem \eqref{p1}, following \cite{YZ99}, we impose the following standard assumption on its coefficients. 
\begin{assmp}\label{assumption}
It holds that 
\begin{enumerate}
\item $b:\overline\setq\to\rn$; ~~  $f: \overline\setq\to\R$; ~~  $\sig:\overline\setq\to\R^{N\times m}$;
\item
The functions $b$, $f$ and $\sig$ are right-continuous w.r.t. $t$; 
\item There exist positive constants $\kappa\in (0,1]$ and $K$ such that 
\begin{align*}
|\phi(t,x,c)-\phi(t,\hatx,\hatc)|&\leq K\big(|x-\hatx|+(1+|x|+|\hatx|)|c-\hatc|^{\kappa}\big) 
\end{align*}
holds for all $(t,x,c)$, $(t,x,\hatc)\in\setq$ and $\phi=b$, $f$ and $\sig$; namely, 
the functions $b$, $f$ and $\sig$ are Lipchitz continuous w.r.t. $x$ and locally H\"older continuous w.r.t. $c$.

\item The function $|\phi(\cdot, 0, \cdot)|$ is bounded on $[0,T]\times[\unc,\barc]$ for $\phi=b$, $f$ and $\sig$.
\end{enumerate}
\end{assmp}
Under this assumption, we see $|\phi(t,x,c)|\leq K(1+|x|)$ on $\overline\setq$ for $\phi=b$, $f$ and $\sig$.
Here and hereafter, we use $K$ to denote some positive constant that may be different at each occurrence  and may relay on $T$, $\unc$ and $\barc$, but it is always independent of the choice of $(t,x,c)\in\setq$.

Because the coefficients in the control problem \eqref{p1} are deterministic, we will study the problem via its HJB equation. We first study the properties of its value function.

\section{Basic Properties of the Value Function}\label{BR}
In this section, we give some basic properties for the value function $V$. Recall that the constraint set $\setad[c,\barc]$ is shrinking as $c$ increases, so the value function $V$ defined in \eqref{p1} is non-decreasing w.r.t. $c$. This property will be used frequently in the subsequent analysis without claim. 

The biggest difference between the HJB equation associated with our problem and those associated with classical ones defined in \cite{YZ99} is that a new special argument $c$ is introduced, which makes our HJB equation a variational inequality, rather than parabolic PDEs as in \cite{YZ99}. The new argument is needed to record the maximum historical control value, since our controls are self-path-dependent. 
The HJB equations associated with American options and optimal stopping problems are also variational inequalities, but our one is essential different from them. The former's gradient constraints in the HJB equations are w.r.t. the state argument $x$, by contrast, our one is against the control argument $c$. 
 
Thanks to the Lipchitz continuity condition in Assumption \ref{assumption}, we can get similar continuity properties on the $t$ and $x$ components. For the new variable $c$, we will prove that the value function is Lipchitz continuous w.r.t. it as well. 

We say a function $\Phi: \setq\to \R$ satisfies the standard estimates if there is a positive constant $K$ such that the following holds: 
\begin{align}\label{estimate1}
|\Phi(t,x,c)| &\leq K(1+|x|),\\
|\Phi(t,x,c)-\Phi(\hatt,\hatx,c)|&\leq K\big(|x-\hatx|+(1+|x|+|\hatx|)|t-\hatt|^{\frac{1}{2}}\big),\label{estimate2}\\
|\Phi(t,x,c)-\Phi(t,x,\hatc)| &\leq K(1+|x|)|c-\hatc|^\kappa,\label{estimate3}
\end{align}
for all $(t,x,c),(\hatt,\hatx,c), (t,x,\hatc)\in\setq$.

\begin{lemma}\label{bound of V}
Suppose Assumption \ref{assumption} holds.  
Then the value function $V$ defined in \eqref{p1} satisfies the standard estimates \eqref{estimate1}-\eqref{estimate3}.
\end{lemma}

\begin{proof}
Fix any $(t,x,c)\in\setq$. 
For any $u(\cdot)\in\setad[c,\barc]$, by \cite[Theorem 1.6.16]{YZ99}, we have
\begin{align}\label{3.4}
\BE[t,x]{\sup_{s\in[t,T]}|X(s)| }\leq K(1+|x|).
\end{align}
Thus by Assumption \ref{assumption}, we have 
\begin{align*}
|V(t, x, c)|&\leq\sup_{u(\cdot)\in\setad[c,\barc]} |J(t, x, u(\cdot))|\\
&\leq \sup_{u(\cdot)\in\setad[\unc,\barc]}\BE[t,x]{\int_{t}^{T}|f(s, X(s), u(s))|\ds+|H(X(T))| }\\
&\leq K(1+|x|).
\end{align*}
Hence the first estimate \eqref{estimate1} for $V$ is established. Here we use the lower bound $\unc$ to get the uniform constant $K$.

Next, suppose $ \hatt\in(t,T]$ and $\hatx\in\rn $. For any $u(\cdot)\in\setad[c,\barc]$, let $X$ and $\hatX$ be the corresponding states driven by \eqref{state1} with the initial positions $(t,x)$ and $(\hatt,\hatx)$ respectively. Then also by \cite[Theorem 1.6.16]{YZ99}, we have 
\begin{align}
\BE{\sup_{s\in[\hatt,T]}|X(s)-\hatX(s)|}\leq K\big(|x-\hatx|+(1+|x|+|\hatx|)|t-\hatt|^{\frac{1}{2}}\big).
\end{align}
Together with Assumption \ref{assumption}, it follows that
\begin{align*} 
|J(\hatt,x,u(\cdot))-J(\hatt,\hatx,u(\cdot))| &\leq K|x-\hatx|,
\end{align*}
and 
\begin{align*}
&\quad\;|J(t,x,u(\cdot))-J(\hatt,x,u(\cdot))|\\
&\leq \BE{\Big|\int_t^T f(s,X(s),u(\cdot))\ds-\int_{\hatt}^{T} f(s,\hatX(s),u(\cdot))\ds\Big|}\\
&\leq \BE{\int_{t}^{\hatt} \big|f(s,X(s),u(\cdot))\big|\ds+\int_{\hatt}^{T}\big|f(s,X(s),u(\cdot))-f(s,\hatX(s),u(\cdot))\big|\ds}\\ 
&\leq K(1+|x|)|t-\hatt|+K(1+|x|+|\hatx|)|t-\hatt|^{\frac{1}{2}}\\
&\leq K(1+|x|+|\hatx|)|t-\hatt|^\frac{1}{2},
\end{align*}
where we used $|t-\hatt|\leq T$ to get the last inequality. 
Combining the above inequalities, we conclude that
\begin{align}\label{estimateofJ}
|J(t,x,u(\cdot))-J(\hatt,\hatx,u(\cdot))|\leq K\big(|x-\hatx|+(1+|x|+|\hatx|)|t-\hatt|^{\frac{1}{2}}\big).
\end{align}
Then the second estimate \eqref{estimate2} for $V$ follows from above estimate easily. 

For any $\hatc\in (c,\barc]$ and any control $u(\cdot)\in\setad[c,\barc]$, let $\hatu(\cdot)=u(\cdot)\vee\hatc$. Then $\hatu(\cdot)\in\setad[\hatc,\barc]$ and $|\hatu(\cdot)-u(\cdot)|\leq |\hatc-c|$.
Let $X$ and $\hatX$ be the corresponding state trajectories driven by \eqref{state1} under the controls $u(\cdot)$ and $\hatu(\cdot)$ with the initial position $(t,x)$. Then we have 
\begin{align*}
&X(t)=x+\int_t^T f(s,X(s),u(s))\ds+\int_t^T \sigma(s,X(s),u(s))\dw_{s}, \\
&\hatX(t)=x+\int_t^T f(s,\hatX(s),\hatu(s))\ds+\int_t^T \sigma(s,\hatX(s),\hatu(s))\dw_{s}. 
\end{align*}
The above implies that 
\begin{align*}
\BE{|X(t)-\hatX(t)|}\leq \BE{\int_t^T\big| f(s,X(s),u(s))- f(s,\hatX(s),\hatu(s)\big|\ds}. 
\end{align*}
Then under Assumption \ref{assumption}, it leads to
\begin{align*}
&\quad\; \BE{|X(t)-\hatX(t)|}\\
&\leq \BE{\int_t^T K|X(s)-\hatX(s)|+K\big(1+|X(t)|+|\hatX(t)|\big)|u(s)-\hatu(s)|^\kappa\ds}\\
&\leq K\int_t^T \BE{|X(s)-\hatX(s)|}\ds+K|\hatc-c|^\kappa\BE{\int_t^T\big(1+|X(s)|+|\hatX(s)|\big)\ds}\\
&\leq K\int_t^T \BE{|X(s)-\hatX(s)|}\ds+K(1+|x|)|\hatc-c|^\kappa,
\end{align*}
where we used \eqref{3.4} to get the last inequality.  
By Gronwall's inequality, the above leads to 
\begin{align}\label{estimatexx}
\BE{|X(t)-\hatX(t)|}\leq K(1+|x|)|\hatc-\barc|^\kappa e^{K(T-t)}\leq K(1+|x|)|\hatc-c|^\kappa.
\end{align}
Thanks to \eqref{estimatexx} and the Lipchitz continuity of the coefficients, it then follows
\begin{align*} 
&\quad\; |J(t,x,u(\cdot))-J(t,x,\hatu(\cdot))|\\
&\leq \BE{\Big|\int_t^T f(s,X(s),u(s))\ds-\int_{t}^{T} f(s,\hatX(s),\hatu(s))\ds\Big|}\\
&\leq \BE{\int_t^T \big|f(s,X(s),u(s))-f(s,\hatX(s),\hatu(s))\big|\ds}\\
&\leq \BE{\int_t^T K\Big( \big|X(s)-\hatX(s)\big|+ \big|u(s)-\hatu(s)\big|\Big)\ds}\\
&\leq K\int_t^T \Big( \BE{\big|X(s)-\hatX(s)\big|}+ |c-\hatc|\Big)\ds\\
&\leq K(1+|x|)|c-\hatc|^\kappa,
\end{align*}
where we used $|c-\hatc|\leq \barc-\unc$ to get the last inequality.
Now we conclude 
\begin{align} 
V(t,x,\hatc)\leq J(t,x,\hatu(\cdot))\leq K(1+|x|)|c-\hatc|^\kappa+J(t,x,u(\cdot)).
\end{align}
Since $u(\cdot)$ is arbitrarily chosen in $\setad[c,\barc]$, it yields 
\begin{align}
V(t,x,\hatc)\leq K(1+|x|)|c-\hatc|^\kappa+ V(t,x,c).
\end{align}
This, together with $V(t,x,c)\leq V(t,x,\hatc)$, implies the estimate \eqref{estimate3} for $V$ holds. 
\end{proof}

We now prove the following dynamic programming principle for the new control problem with self-path-dependent constraint.

\begin{thm}[Dynamic programming principle]\label{principle of optimality}
Assume Assumption \ref{assumption} holds. Then for any $(t,x,c)\in \setq$ and $\hatt\in [t,T]$, it holds that 
\begin{align} \label{ddp1}
V(t,x,c)=\inf_{u(\cdot)\in\setad[c,\barc]}\BE[t,x]{\int_t^{\hatt}f(s,X(s),u(s))\ds
+V(\hatt, X(\hatt), u(\hatt))}.
\end{align}
\end{thm}
\begin{proof}
Denote the right-hand side of \eqref{ddp1} by $\overline{V}(t,x,c)$. 
Then, for any $u(\cdot)\in\setad[c,\barc]$, 
\begin{align*}
J(t,x,u(\cdot)) &=\BE[t,x]{\int_t^Tf(s,X(s),u(s))\ds+H(X(T))}\\
&=\BE[t,x]{\int_t^{\hatt}f(s,X(s),u(s))\ds}\\ 
&\quad\;+\BE[t,x]{\BE{\int_{\hatt}^{T}f(s,X(s),u(s))\ds+H(X(T))\;\Big|\;X(\hatt),u(\hatt)}}\\
&\geq \BE[t,x]{\int_t^{\hatt}f(s,X(s),u(s))\ds} +\BE[t,x]{V(\hatt, X(\hatt), u(\hatt))}\\
&\geq \overline{V}(t,x,c),
\end{align*}
where we use the fact that $u(\cdot)\in\setad[u(\hatt),\barc]$ on the time interval $[\hatt,T]$ to get the second inequality. Hence, we get $V(t,x,c)\geq \overline{V}(t,x,c)$ as $u(\cdot)\in\setad[c,\barc]$ is arbitrarily chosen.

On the other hand, for any constant $\epsilon>0$, by \citelem{bound of V} and estimate \eqref{estimateofJ}, there exists a constant $\delta=\delta(\epsilon)>0$ such that whenever $|x-\hatx|\leq\delta$, it holds that 
\begin{align}\label{approx1}
|J(\hatt,x,u(\cdot))-J(\hatt,\hatx,u(\cdot))|+|V(\hatt,x,c)-V(\hatt, \hatx, c)|<\epsilon,
\end{align}
for any control $u(\cdot)\in\setad[c,\barc]$.
Let $\{D_j\}_{j\ge1}$ be a Borel partition of $\rn$ with diameter diam $(D_j)<\delta$ and fix a point $x_j\in D_j$ for each $j$. Then for each $j$, there is a control $u_j(\cdot)\in\setad[c,\barc]$ such that $J(\hatt,x_j,u_j(\cdot)) < V(\hatt,x_j,c)+\epsilon$. 
Hence for any $x\in D_j$, we obtain from \eqref{approx1} that
\begin{align}\label{eachjestimate}
J(\hatt,x,u_j(\cdot))< J(\hatt,x_j,u_j(\cdot))+\epsilon < V(\hatt,x_j,c)+2\epsilon < V(\hatt,x,c)+3\epsilon.
\end{align} 
Given any control $u(\cdot)\in\setad[c,\barc]$, let $X$ represent the corresponding state trajectory driven by \eqref{state1} with the initial position $(t,x)$. By \eqref{eachjestimate} and the monotonicity of $V$ w.r.t. $c$, 
\begin{align} 
J(\hatt,X(\hatt),u_{j}(\cdot))\idd{X(\hatt)\in D_{j}}
&\leq (V(\hatt, X(\hatt), c)+3\epsilon )\idd{X(\hatt)\in D_{j}}\nn\\
&\leq (V(\hatt, X(\hatt),u(\hatt))+3\epsilon )\idd{X(\hatt)\in D_{j}}.\label{approx2}
\end{align}
Define
\begin{align}
\hatu(s)=\begin{cases}
u(s) &\quad \mbox{if }s\in[t,\hatt);\\
u_{j}(s) &\quad \mbox{if $s\in[\hatt,T]$ and $X(\hatt)\in D_{j}$}.
\end{cases}
\end{align}
Clearly, $\hatu(s)$ is right-continuous and $\hatu(\cdot)\in\setad[c,\barc]$.
Applying \eqref{approx2}, we get
\begin{align*}
V(t,x,c)&\leq J(t,x,\hatu(\cdot)) =\BE{\int_t^Tf(s,X(s),\hatu(s))\ds+H(X(T))}\\ 
&=\BE{\int_t^{\hatt}f(s,X(s),\hatu(s))\ds+\int_{\hatt}^Tf(s,X(s),\hatu(s))\ds+H(X(T))}\\
&=\BE{\int_t^{\hatt}f(s,X(s),u(s))\ds+\sum_{j\geq 1}J(\hatt,X(\hatt),u_{j}(\cdot))\idd{X(\hatt)\in D_{j}}}\\
&\leq \BE{\int_t^{\hatt}f(s,X(s),u(s))\ds+\sum_{j\geq 1} \big(V(\hatt, X(\hatt),u(\hatt))+3\epsilon\big)\idd{X(\hatt)\in D_{j}} }\\
&=\BE{\int_t^{\hatt}f(s,X(s),u(s))\ds+V(\hatt, X(\hatt),u(\hatt))}+3\epsilon.
\end{align*}
By taking the infimum over $u(\cdot)\in\setad[c,\barc]$ and sending $\epsilon\to 0$, we get $V(t,x,c)\leq \overline{V}(t,x,c)$, which completes the proof. 
\end{proof}

\section{HJB Equation and Viscosity Solution}\label{VS}
In this section, we will prove the main result of this paper, that is, the value function $V$ is the unique viscosity solution of the related HJB equation. 

\subsection{HJB Equation}
We now introduce the Hamilton-Jacobi-Bellman equation associated with the problem \eqref{p1}.  

To begin, we define the following function:
\begin{align}\label{defofG}
\mathcal{G}(t,x,u,p,P)=\frac{1}{2}\mathrm{tr}\big(P\sigma(t,x,u)\sigma(t,x,u)^T\big)+\langle b(t,x,u),p\rangle-f(t,x,u),
\end{align}

We next determine the boundary values of the value function $V$. It is clear that when $t=T$, we have $V(T,x,c)=H(x)$ for all $x\in\rn$ and $c\in[\unc,\barc]$. On the other hand, since every admissible control must be non-decreasing, there is only one admissible (thus optimal) control $u(\cdot)\equiv\barc$ in $\setad[\barc,\barc]$, so 
\begin{align*}
V(t, x, \barc)&=\inf_{u(\cdot)\in\setad[\barc,\barc]}J(t,x,u(\cdot))=G(t,x),
\end{align*}
where
\begin{align}\label{defG}
G(t,x):=J(t,x, \barc),~~(t,x)\in[0,T]\times \rn.
\end{align}

Recall that the value function $V$ defined in \eqref{p1} is non-decreasing w.r.t. $c$, hence, compared with the classic control problems without self-path-dependent constraint, the HJB equation associated with \eqref{p1} shall introduce a monotonic condition w.r.t. the argument $c$, that is, a gradient constraint. This leads to the following new type of HJB equation: 
\begin{align}\label{HJB}
\begin{cases}
\max\big\{-v_t+\mathcal{G}(t,x,c,-v_x,-v_{xx}),-v_c\big\}=0,\\
v(T,x,c)=H(x),\\
v(t,x,\barc)=G(t,x), ~~\forall\; (t,x,c)\in \setq,
\end{cases}
\end{align}
where $G$ is defined in \eqref{defG}.

The HJB equation \eqref{HJB} is a new type of variational inequality in the sense that gradient constraint $-v_c\leq 0$ is put on the argument $c$, whereas the differential operator $\mathcal{L}$ does not involve the derivatives of $v$ w.r.t. the argument $c$. By contrast, for the typical variational inequalities in financial engineering, such as the Black-Scholes PDEs for American options, both the gradient constraint and the differential operator are put on the same argument. The above difference leads to fairly large difficulties to the study of the 
HJB equation \eqref{HJB}. Until very recently, one can only find viscosity solutions in the linear case (namely $b$, $f$ and $\sig$ are linear w.r.t. $(x,u)$); see, e.g., \cite{albrecher2017optimal}-\cite{albrecher2023optimal}. Using a partial differential equation method to study their HJB equations, \cite{guan2024optimal1} and \cite{guan2023optimal2} found stronger solutions to their specific linear ratcheting models. However, there are no works in the literature to address nonlinear problems. This paper extends the study to the general Lipchitz continuity case.
To prove the uniqueness, we also need the following assumption for technical reason.
\begin{assmp}\label{assumption2}
There exist positive constants $\kappa\in (0,1]$ and $K$ such that 
\begin{align*}
|\phi(t,x,c)-\phi(\hatt,x,c)|&\leq K(1+|x|)|t-\hatt|^{\kappa}\label{estimate5} 
\end{align*}
holds for all $(t,x,c)$, $(\hatt,x,c)\in\setq$ and $\phi=b$, $f$, $\sigma$, namely, the functions $b$, $f$ and $\sigma$ are locally H\"older continuous w.r.t. $t$.
\end{assmp}
The main result of this paper is stated as follows. 

\begin{thm}\label{mainthm}
Suppose Assumption \ref{assumption} holds. Then the value function $V$ defined in \eqref{p1} is a viscosity solution to the HJB equation \eqref{HJB} in the class of functions that satisfy the standard estimates \eqref{estimate1}-\eqref{estimate3}. In addition, it is the unique one provided that Assumption \ref{assumption2} holds.
\end{thm}

\citethm{mainthm} is a consequence of \citethm{existence} and \citethm{comparison}.
The former proves the existence in Section \ref{sec:existence} and the later establishes a omparison theorem
and thus confirms the uniqueness in Section \ref{sec:unique}.

\subsection{Definition of Viscosity Solution}
The precise definition of viscosity solution to \eqref{HJB} is given as follows. 
\begin{defn}
The viscosity subsolution, viscosity supersolution and viscosity solution to \eqref{HJB} are defined as follows. 
\begin{description} 
\item [Viscosity subsolution:] An upper semicontinuous function $\underline{v}: \setq\to\R$ is a called viscosity subsolution to \eqref{HJB} if for any point $(t,x,c)\in\setq$ with $t<T$, 
it holds that $\underline{v}(T,x,c)\leq H(x) $, $\underline{v}(t,x,\barc)\leq G(t,x)$ 
and $$\max\big\{-\phi_t+\mathcal{G}(t,x,c,-\phi_x,-\phi_{xx}),-\phi_c\big\}\big|_{(t,x,c)}\leq 0$$
for any function $\phi \in C^{1,2,1}(Q)$ such that the maximum value 0 of the function $\underline{v}-\phi$ is achieved at the point $(t,x,c)$.

\item [Viscosity supersolution:] A lower semicontinuous function $\bar{v}:\setq\to\R$ is a called viscosity supersolution to \eqref{HJB} if for any point $(t,x,c)\in\setq$ with $t<T$,
it holds that $\bar{v}(T,x,c)\geq H(x)$, $\barv(t,x,\barc)\geq G(t,x)$
and $$\max\big\{-\phi_t+\mathcal{G}(t,x,c,-\phi_x,-\phi_{xx}),-\phi_c\big\}\big|_{(t,x,c)}\geq 0$$
for any function $\phi \in C^{1,2,1}(Q)$ such that the minimum value 0 of the function $\bar{v}-\phi$ is achieved at the point $(t,x,c)$. 

\item [Viscosity solution:] A function $v: \setq\to \R$ is called a viscosity solution to \eqref{HJB} if it is 
both a viscosity subsolution and a viscosity supersolution to it. 
\end{description}

The function $\phi$ in the above definitions is called a viscosity subsolution (resp. supersolution or solution) test function for $v$ at the point $(t,x,c)$.
\end{defn}

\subsection{Existence of Viscosity Solution}\label{sec:existence}
We now prove that there is at least one viscosity solution to the HJB equation \eqref{HJB}. Indeed, the value function $V$ defined in \eqref{p1} is such a one. 
\begin{thm}[Existence]\label{existence}
Suppose Assumption \ref{assumption} holds. Then the value function $V$ defined in \eqref{p1} is a viscosity solution to the HJB equation \eqref{HJB} in the class of functions that satisfy the standard estimates \eqref{estimate1}-\eqref{estimate3}.
\end{thm}
\begin{proof}
It is already shown that the value function satisfies the estimates \eqref{estimate1}-\eqref{estimate3} in \citelem{bound of V} and the boundary conditions in \eqref{HJB}.
Also thanks to \citelem{bound of V}, we know $V$ is continuous on $\setq$.

Now we show that the value function $V$ is a viscosity subsolution to \eqref{HJB}. 
Fix any point $(t,x,c)\in\setq$ with $t<T$. For any $c\leq \hatc\leq \barc$, any subsolution test function $\phi\in C^{1,2,1}(\setq)$ at $(t,x,c)$, we have 
$$V(t,x,\hatc)- \phi(t,x,\hatc)\leq V(t,x,c)-\phi(t,x,c)=0,$$
so 
$$\phi(t,x,c)-\phi(t,x,\hatc)\leq V(t,x,c)-V(t,x,\hatc)\leq 0,$$
where the last inequality is due to the monotonicity of $V$ w.r.t. $c$. It clearly follows $\phi_c\geq 0$. 

Now fix the control $u(\cdot)\equiv c$ and let $X^c(t)$ represent the state trajectory in \eqref{state1} under it. For any $t<\hatt\leq T$, by It\^o's Lemma and \citethm{principle of optimality}, we have
\begin{align*}
0&\leq\frac{\BE[t,x]{{V(t,x,c)-\phi(t,x,c)-V(\hatt,X^c(\hatt),c)+\phi(\hatt, X^c(\hatt),c)}}}{\hatt-t}\\
&\leq \frac{\BE[t,x]{\int_t^{\hatt}f(s,X^{c}(s),c)\ds-\phi(t,x,c)+\phi(\hatt, X^c(\hatt),c)}}{\hatt-t}\\
&= \frac{\BE[t,x]{\int_t^{\hatt}(f+\frac{1}{2}\sigma^2\phi_{xx}+b\phi_x-\phi_t)(s,X^{c}(s),c)\ds}}{\hatt-t}\\
&\xrightarrow{}\phi_t-\mathcal{G}(t,x,c,-\phi_x,-\phi_{xx})\big|_{(t,x,c)} ~~\mbox{as~~$\hatt\rightarrow t$},
\end{align*} 
where we used the right-continuity of $f$, $b$ and $\sig$ w.r.t. $t$ in the last step.
Summarizing up above, we conclude $$\max\big\{-\phi_t+\mathcal{G}(t,x,c,-\phi_x,-\phi_{xx}),-\phi_c\big\}\big|_{(t,x,c)}\leq 0.$$
This shows that $V$ is a viscosity subsolution to \eqref{HJB}.

On the other hand, suppose $V-\phi$ attains its global minimum value 0 at some point $(t,x,c)\in\setq$ with $t<T$. For any $\epsilon>0$ and $t<\hatt\leq T$, by \citethm{principle of optimality}, we can find a control $u(\cdot)\in\setad[c,\barc]$ such that 
\begin{align*}
0&\geq \BE[t,x]{V(t,x,c)-\phi(t,x,c)-V(\hatt,X(\hatt),c)+\phi(\hatt,X(\hatt),c)} \\
&\geq -\epsilon(\hatt-t)+\BE[t,x]{\int_t^{\hatt}f(s,X(s),u(s))ds+\phi(\hatt,X(\hatt),c)-\phi(t,x,c)}
\end{align*}
Dividing by $(\hatt-t)$ and applying It\^o's Lemma to the process $\phi(t,x(t),c)$, we get
\begin{align}\label{super}
\epsilon&\geq 
\frac{\BE[t,x]{\int_t^{\hatt}(f+\frac{1}{2}\sigma^2\phi_{xx}+b\phi_x+\phi_t)(s,X(s), u(s))\ds}}{\hatt-t}.
\end{align}
We rewrite the above by using symbol $\mathcal{G}$ in \eqref{defofG} and it becomes
\begin{align*} 
\epsilon&\geq 
\frac{\BE[t,x]{\int_t^{\hatt}(\phi_t-\mathcal{G}(s,X(s), u(s),-\phi_x,-\phi_{xx}))\ds}}{\hatt-t}.
\end{align*}
Thanks to the right-continuity of $f$, $b$ and $\sig$ w.r.t. $t$, by Fatou's lemma, 
sending $\hatt\rightarrow t+$ in above gives
\[\epsilon \geq \phi_t-\mathcal{G}(t,x,c,-\phi_x,-\phi_{xx})\big|_{(t,x,c)}.\]
Since $\epsilon>0$ is arbitrary chosen, it follows 
\begin{align*}
\max\big\{-\phi_t+\mathcal{G}(t,x,c,-\phi_x,-\phi_{xx}),-\phi_c\big\}\big|_{(t,x,c)}\geq -\phi_t+\mathcal{G}(t,x,c,-\phi_x,-\phi_{xx})\big|_{(t,x,c)}\geq0,
\end{align*} 
so $V$ is also a viscosity supersolution to \eqref{HJB}.

The proof of existence is complete.
\end{proof}

\subsection{Uniqueness of Viscosity Solution}\label{sec:unique}
The proof of the uniqueness of the viscosity solution to \eqref{HJB} is much more involved than that of the existence. We need to approximate the viscosity solution by semiconvex and semiconcave functions.
\begin{defn}
The semiconcave and semiconvex functions are defined as follows.
\begin{description}
\item [Semiconcave function:]
A function $\phi:\R^n\to\R$ is called semiconcave if there exists a constant $K\geq 0$ such that $x\mapsto\phi(x)-K|x|^2$ is concave.
\item [Semiconvex function:] 
A function $\phi:\R^n\to\R$ is called semiconvex if the function $-\phi$ is semiconcave.
\end{description}
\end{defn}

\begin{lemma}\label{lemma1}
Suppose a function $v: \setq\to\R$ satisfy the estimates \eqref{estimate1}-\eqref{estimate3}. For any constant $0<\gamma\leq 1$, define
\begin{align}\label{defsemiconvex}
v^\gamma(t,x,c)=\sup_{(\widetilde{t},\widetilde{x},\widetilde{c})\in\setq}\Big\{v(\widetilde{t},\widetilde{x},\widetilde{c})-\frac{1}{2\gamma^2}[|t-\widetilde{t}|^2+|x-\widetilde{x}|^2+|c-\widetilde{c}|^2]\Big\},~~(t,x,c)\in\setq.
\end{align}
Then for any $0<\gamma\leq 1$, $(t,x,c)\in\setq$, there exists a $(\hatt, \hatx,\hatc)\in\overline\setq$
such that
\begin{align}\label{existence}
v^\gamma(t,x,c)=v(\hatt,\hatx,\hatc)-\frac{1}{2\gamma^2}[|t-\hatt|^2+|x-\hatx|^2+|c-\hatc|^2],
\end{align}
where $\overline\setq=[0,T]\times\rn\times [\unc,\barc]$ is closure of $\setq$. 
Moreover, there is a positive constant $K$, which is independent of $(t,x,c)$, $(\hatt,\hatx,\hatc)$, and $0<\gamma\leq 1$, such that the following estimate holds:
\begin{align} 
\frac{1}{2\gamma^2}[|t-\hatt|^2+|x-\hatx|^2+|c-\hatc|^2]&\leq K\gamma^\frac{2\kappa'}{2-\kappa'}(1+|x|)^\frac{2}{2-\kappa'},\label{hatthatxhatcbound}
\end{align}
where $\kappa'=\kappa\wedge\frac{1}{2}$. 

As a consequence, we have the estimate:
\begin{align}\label{vvgammadistance}
0\leq v^\gamma(t,x,c)-v(t,x,c)\leq K\gamma^\frac{2\kappa'}{2-\kappa'}(1+|x|)^\frac{2}{2-\kappa'}.
\end{align}
\end{lemma}
\begin{proof}
First, from \eqref{estimate1} and the definition of $v^\gamma$, there exists a $(\hatt,\hatx,\hatc)\in\overline\setq$ such that \eqref{existence} holds.
Next, \eqref{defsemiconvex} and \eqref{existence} imply that 
\begin{align}\label{vvgammainequality}
v(t,x,c)\leq v^\gamma(t,x,c)\leq v(\hatt,\hatx,\hatc).
\end{align} 
It then follows from \eqref{estimate1} and \eqref{existence} that 
\begin{align*}
|x-\hatx|^2&\leq 2\gamma^2[v(\hatt,\hatx,\hatc)-v^\gamma(t,x,c)]\leq2\gamma^2[v(\hatt,\hatx,\hatc)-v(t,x,c)]\\&\leq2K\gamma^2(2+|x|+|\hatx|)\leq4K (1+|x|)+2K |x-\hatx|,
\end{align*} 
where we used $0<\gamma\leq 1$ to get the last inequality.
After completing square, it follows
\begin{align*}
(|x-\hatx|-K)^2& \leq4K(1+|x|)+K^{2},
\end{align*}
so
\begin{align*}
|\hatx-x|&\leq K+\sqrt{ K^2+4 K(1+|x|)} \leq K(1+|x|).
\end{align*} 
So we conclude that 
\begin{align}
|\hatx-x|&\leq K(1+|x|),\label{hatxxbound}
\end{align}
and
\begin{align}
|\hatx| &\leq K(1+|x|).\label{hatxbound}
\end{align}

By \eqref{estimate2}, \eqref{estimate3} and \eqref{existence}, we have 
\begin{align*}
&\quad\;\frac{1}{2\gamma^2}[|t-\hatt|^2+|x-\hatx|^2+|c-\hatc|^2]\\
&=v(\hatt,\hatx,\hatc)-v^\gamma(t,x,c)\nn\\
&\leq v(\hatt,\hatx,\hatc)-v(t,x,c)\nn\\ 
&\leq K\big(|x-\hatx|+(1+|x|+|\hatx|)|t-\hatt|^\frac{1}{2}+(1+|x|+|\hatx|)|c-\hatc|^\kappa\big). \end{align*} 
For simplifying the notation, we let $\kappa'=\kappa\wedge\frac{1}{2}$. By \eqref{hatxxbound} and \eqref{hatxbound} and above, 
\begin{align*}
&\quad\;\frac{1}{2\gamma^2}[|t-\hatt|^2+|x-\hatx|^2+|c-\hatc|^2]\\ 
&\leq K(1+|x|)\big(|x-\hatx|^{\kappa'}+|t-\hatt|^{\kappa'}+|c-\hatc|^{\kappa'}\big)\\
&\leq K(1+|x|)\big(|x-\hatx|^2+|t-\hatt|^2+|c-\hatc|^2\big)^\frac{\kappa'}{2},
\end{align*}
which leads to the estimate 
\begin{align}\label{bound1}
|t-\hatt|^2+|x-\hatx|^2+|c-\hatc|^2\leq K\gamma^\frac{4}{2-\kappa'}(1+|x|)^\frac{2}{2-\kappa'},
\end{align} 
which means \eqref{hatthatxhatcbound} holds.

By \eqref{hatthatxhatcbound}, we get 
\begin{align*}
v^\gamma(t,x,c)-v(t,x,c)&\leq v(\hatt,\hatx,\hatc)-v(t,x,c) \leq K\gamma^\frac{2\kappa'}{2-\kappa'}(1+|x|)^\frac{2}{2-\kappa'}.
\end{align*} 
The above estimate leads to \eqref{vvgammadistance}. The proof is complete.
\end{proof}

Lemma \ref{lemma1} gives some estimates on the distance between $v^\gamma$ and $v$. To prove the uniqueness, we also need some estimates on $v^\gamma$ itself.
\begin{lemma}\label{lemma1.1}
Suppose $0<\gamma\leq 1$, the functions $v$ and $v^\gamma$ are the same as in Lemma \ref{lemma1}. Then the function $v^\gamma$ is semiconvex w.r.t. $(t,x,c)$ and satisfies the following estimates:
\begin{align}
|v^\gamma(t_1,x_1,c_1)| &\leq  K(1+|x_1|),\label{vupgammabound}\\
|v^\gamma(t_1,x_1,c_1)-v^\gamma(t_2,x_2,c_2)| &\leq  K\gamma^\frac{2\kappa'}{2-\kappa'}(1+|x_1|)^\frac{2}{2-\kappa'}+K\gamma^\frac{2\kappa'}{2-\kappa'}(1+|x_2|)^\frac{2}{2-\kappa'}\nn\\
&\quad\;+K\big(|x_1-x_2|+(1+|x_1|+|x_2|)|t_1-t_2|^{\frac{1}{2}}\nn\\
&\quad\;+(1+|x_1|+|x_2|)|c_1-c_2|^\kappa\big),\label{vgammadistance}
\end{align} 
for all $(t_1,x_1,c_1)$, $(t_2,x_2,c_2)\in\setq$ and $0<\gamma\leq 1$.
\end{lemma}
\begin{proof}
By \eqref{vvgammainequality}, \eqref{estimate1} and \eqref{hatxbound}, we obtain 
\eqref{vupgammabound}.

By \eqref{estimate2} and \eqref{vvgammadistance}, we have
\begin{align*}
&\quad\;|v^\gamma(t_1,x_1,c_1)-v^\gamma(t_2,x_2,c_2)|\nn\\
&\leq |v^\gamma(t_1,x_1,c_1)-v(t_1,x_1,c_1)|+|v(t_1,x_1,c_1)-v(t_2,x_2,c_2)|+|v(t_2,x_2,c_2)-v^\gamma(t_2,x_2,c_2)|\nn\\
&\leq K\gamma^\frac{2\kappa'}{2-\kappa'}(1+|x_1|)^\frac{2}{2-\kappa'}+K\gamma^\frac{2\kappa'}{2-\kappa'}(1+|x_2|)^\frac{2}{2-\kappa'}\nn\\
&\quad\;+K\big(|x_1-x_2|+(1+|x_1|+|x_2|)|t_1-t_2|^{\frac{1}{2}}+(1+|x_1|+|x_2|)|c_1-c_2|^\kappa\big),
\end{align*}
which gives \eqref{vgammadistance}.

It is left to prove that $v^\gamma(t,x,c)+\frac{1}{2\gamma^2}(|t|^2+|x|^2+|c|^2)$ is convex w.r.t. $(t,x,c)$.
Denote
\begin{align*}
\phi(t,x,c)&=v^\gamma(t,x,c)+\frac{1}{2\gamma^2}(|t|^2+|x|^2+|c|^2)\\
&=\sup_{(\widetilde{t},\widetilde{x},\widetilde{c})\in\setq}\Big\{v(\widt,\widx,\widc)-\frac{1}{2\gamma^2}\big[(\widt(2t-\widt)+\widx(2x-\widx)+\widc(2c-\widc)\big]\Big\}.
\end{align*}
Then for any constant $\lambda\in (0,1)$ and $(t_1,x_1,c_1)$, $(t_2,x_2,c_2)\in\setq$, we have 
\begin{align*}
&\quad\;\phi(\lambda t_1+(1-\lambda)t_2,\lambda x_1+(1-\lambda)x_2,\lambda c_1+(1-\lambda)c_2)\\
&=\sup_{(\widt,\widx,\widc)\in\setq}\Big\{v(\widt,\widx,\widc)-\frac{1}{2\gamma^2}[(\widt(2\lambda t_1+2(1-\lambda)t_2-\widt)\\
&\quad\;+\widx(2\lambda x_1+2(1-\lambda)x_2-\widx)+d(2\lambda c_1+2(1-\lambda)c_2-d)]\Big\}\\
&=\sup_{(\widt,\widx,\widc)\in\setq}\Big\{\lambda v(\widt,\widx,\widc)-\lambda\frac{1}{2\gamma^2}[\widt(2t_1-\widt)+\widx(2x_1-\widx)+\widc(2c_1-\widc)]\\
&\quad\;+(1-\lambda)v(\widt,\widx,\widc)-(1-\lambda)\frac{1}{2\gamma^2}[\widt(2t_2-\widt)+\widx(2x_2-\widx)+\widc(2c_2-\widc)]\Big\}\\
&\leq \lambda \sup_{(\widt,\widx,\widc)\in\setq}\Big\{ v(\widt,\widx,\widc)- \frac{1}{2\gamma^2}[\widt(2t_1-\widt)+\widx(2x_1-\widx)+\widc(2c_1-\widc)]\Big\}\\
&\quad\;+(1-\lambda) \sup_{(\widt,\widx,\widc)\in\setq}\Big\{ v(\widt,\widx,\widc)-\frac{1}{2\gamma^2}[\widt(2t_2-\widt)+\widx(2x_2-\widx)+\widc(2c_2-\widc)]\Big\}\\
&= \lambda\phi(t_1,x_1,c_1)+(1-\lambda)\phi(t_2,x_2,c_2).
\end{align*}
This proved that $v^\gamma$ is semiconvex w.r.t. $(t,x,c)$.
\end{proof}

Similarly to Lemma \ref{lemma1} and Lemma \ref{lemma1.1}, we can give the corresponding concave estimate function of $v$.
\begin{lemma}\label{lemma2}
Suppose a function $v: \setq\to\R$ satisfy the estimates \eqref{estimate1}-\eqref{estimate3}. For any constant $0<\gamma\leq 1$, define
\begin{align}\label{defsemiconcave}
v_\gamma(t,x,c)=\inf_{(\widt,\widx,\widc)\in\setq}\Big\{v(\widt,\widx,\widc)+\frac{1}{2\gamma^2}[|t-\widt|^2+|x-\widx|^2+|c-\widc|^2]\Big\},~~(t,x,c)\in\setq.
\end{align}
Then for any $0<\gamma\leq 1$, $(t,x,c)\in\setq$, there exists a $(\hatt, \hatx,\hatc)\in\overline\setq$ such that
\begin{align}\label{existence2}
v_\gamma(t,x,c)=v(\hatt,\hatx,\hatc)+\frac{1}{2\gamma^2}[|t-\hatt|^2+|x-\hatx|^2+|c-\hatc|^2].
\end{align}
Moreover, there is a positive constant $K$, which is independent of $(t,x,c)$, $(\hatt,\hatx,\hatc)$ and $0<\gamma\leq 1$,  
such that the estimate \eqref{hatthatxhatcbound} holds. 
As a consequence,
\begin{align}\label{vvgammadistance2}
0\leq v(t,x,c)-v_\gamma(t,x,c)\leq K\gamma^\frac{2\kappa'}{2-\kappa'}(1+|x|)^\frac{2}{2-\kappa'}.
\end{align}
\end{lemma}
\begin{lemma}\label{lemma2.1}
Suppose $0<\gamma\leq 1$, the functions $v$ and $v_\gamma$ are the same as in Lemma \ref{lemma2}. Then the function $v_\gamma$ is semiconcave w.r.t. $(t,x,c)$ and satisfies the following estimates:
\begin{align}
|v_\gamma(t_1,x_1,c_1)| &\leq K(1+|x_1|),\label{vgammadownbound}\\
|v_\gamma(t_1,x_1,c_1)-v_\gamma(t_2,x_2,c_2)| &\leq K\gamma^\frac{2\kappa'}{2-\kappa'}(1+|x_1|)^\frac{2}{2-\kappa'}+K\gamma^\frac{2\kappa'}{2-\kappa'}(1+|x_2|)^\frac{2}{2-\kappa'}.\nn\\
&\quad\;+K\big(|x_1-x_2|+(1+|x_1|+|x_2|)|t_1-t_2|^{\frac{1}{2}}\nn\\
&\quad\;+(1+|x_1|+|x_2|)|c_1-c_2|^\kappa\big),\label{vgammadistance}
\end{align} 
for all $(t_1,x_1,c_1)$, $(t_2,x_2,c_2)\in\setq$ and $0<\gamma\leq 1$.
\end{lemma}

\begin{proof}
The proofs of Lemma \ref{lemma2} and Lemma \ref{lemma2.1} are similar to that of Lemma \ref{lemma1} and Lemma \ref{lemma1.1}, so we omit them. 
\end{proof}

Combining \eqref{vvgammadistance} and \eqref{vvgammadistance2}, we can get an important estimate on the distance between $v^\gamma$ and $v_\gamma$: 
\begin{align}\label{updowngammadistance}
0\leq v^\gamma(t,x,c)-v_\gamma(t,x,c)\leq K\gamma^\frac{2\kappa'}{2-\kappa'}(1+|x|)^\frac{2}{2-\kappa'}.
\end{align}

Next, we define two functions which will be used to establish a new viscosity solution combined with semiconcave and semiconvex approximations.
For any $(t,x,c)\in\setq$, we define two functions: 
\begin{align}\label{defupgamma}
&\mathcal{L}^\gamma (t,x,c,p,P)=\inf_{(\widt,\widx,\widc)\in\setc(t,x,c)}\mathcal{G}(\widt,\widx,\widc,p,P),
\end{align}
and
\begin{align}\label{defdowngamma}
&\mathcal{L}_\gamma (t,x,c,p,P)=\sup_{(\widt,\widx,\widc)\in\setc(t,x,c)}\mathcal{G}(\widt,\widx,\widc,p,P),
\end{align}
where the set $\setc(t,x,c)$ is defined as 
\begin{align*}
\setc(t,x,c)=\Big\{(\hatt,\hatx,\hatc)\in\setq~\Big|~\frac{1}{2\gamma^2}[|t-\hatt|^2+|x-\hatx|^2+|c-\hatc|^2]\leq K\gamma^\frac{2\kappa'}{2-\kappa'}(1+|x|)^\frac{2}{2-\kappa'}\Big\},
\end{align*}
where $K$ is any sufficiently large constant so that all the estimates in Lemmas \ref{lemma1}, \ref{lemma1.1} \ref{lemma2}, and \ref{lemma2.1} hold.  

We now derive the HJB equations for the semiconvex and semiconcave approximations $v^\gamma$ and $v_{\gamma}$. 
\begin{lemma}\label{gammasupersubviscosity}
Suppose Assumption \ref{assumption} hold and $0<\gamma\leq 1$.
If $v\in C(\setq)$ is a viscosity subsolution to \eqref{HJB}, then the function $v^\gamma$ defined in Lemma \ref{lemma1} is a viscosity subsolution of the following
\begin{align}\label{subHJB}
\begin{cases}
&\max\{-v_t+\mathcal{L}^\gamma(t,x,c,-v_x,-v_{xx}), -v_c\}=0, \\
&v\big|_{t=T}=v^\gamma(T,x,c),\\
&v\big|_{c=\barc}=v^\gamma(t,x,\barc), ~~~~(t,x,c)\in\setq.
\end{cases}
\end{align}
Likewise, if $v\in C(\setq)$ is a viscosity supersolution to $\eqref{HJB}$, then the function $v_\gamma$ defined in Lemma \ref{lemma2} is a viscosity supersolution of the following:
\begin{align}\label{supHJB}
\begin{cases}
&\max\{-v_t+\mathcal{L}_\gamma(t,x,c,-v_x,-v_{xx}), -v_c\}=0,\\
&v\big|_{t=T}=v_\gamma(T,x,c), \\
&v\big|_{c=\barc}=v_\gamma(t,x,\barc), ~~~~(t,x,c)\in\setq.
\end{cases}
\end{align}
\end{lemma}
\begin{proof}
For any point $(t,x,c)\in\setq$ with $t<T$,
let $\phi$ be a viscosity subsolution test function for $v^{\gamma}$ at the point $(t,x,c)$. 
Suppose $(\hatt,\hatx,\hatc)\in\setq$ satisfies \eqref{existence}. From \eqref{hatthatxhatcbound}, we know $(\hatt,\hatx,\hatc)\in\setc(t,x,c)$.  
Then, whenever $(\tau,\zeta,\eta)\in\setq$, one has 
\begin{align*}
v(\hatt,\hatx,\hatc)-\phi(t,x,c)&=v^\gamma(t,x,c)-\phi(t,x,c)+\frac{1}{2\gamma^2}(|t-\hatt|^2+|x-\hatx|^2+|c-\hatc|^2)\\
&\ge v^\gamma(\tau+t-\hatt,\zeta+x-\hatx,\eta+c-\hatc)-\phi(\tau+t-\hatt,\zeta+x-\hatx,\eta+c-\hatc)\\
&\quad\;+\frac{1}{2\gamma^2}(|t-\hatt|^2+|x-\hatx|^2+|c-\hatc|^2)\\
&\ge v(\tau,\zeta,\eta)-\phi(\tau+t-\hatt,\zeta+x-\hatx,\eta+c-\hatc)
\end{align*}
Set
\begin{align}\label{defofvarphi}
\varphi(\tau,\zeta,\eta)=\phi(\tau+t-\hatt,\zeta+x-\hatx,\eta+c-\hatc).
\end{align} Then
\begin{align}
v(\hatt,\hatx,\hatc)-\varphi(\hatt,\hatx,\hatc)\geq v(\tau,\zeta,\eta)-\varphi(\tau,\zeta,\eta).
\end{align}
Thus $\varphi$ is a viscosity subsolution test function for $v$ at the point $(\hatt,\hatx,\hatc)$. From the definition of the viscosity subsolution, we know that 
\begin{align}\label{equation2}
\max\big\{-\varphi_t+\mathcal{G}(\hatt,\hatx,\hatc,-\varphi_x,-\varphi_{xx}),-\varphi_c\big\}\big|_{(
\hatt,\hatx,\hatc)}\leq 0.
\end{align}
by \eqref{defofvarphi}, we know that 
\begin{align}\label{equation1}
&\phi(t,x,c)=\varphi(\hatt,\hatx,\hatc),\nn\\
&\phi_t(t,x,c)=\varphi_t(\hatt,\hatx,\hatc),\nn\\
&\phi_x(t,x,c)=\varphi_x(\hatt,\hatx,\hatc),\nn\\
&\phi_{xx}(t,x,c)=\varphi_{xx}(\hatt,\hatx,\hatc).
\end{align}
From the definition of $\mathcal{L}^\gamma(t,x,c,-\phi_x,-\phi_{xx})$ in \eqref{defupgamma}, we have
\begin{align*}
-\phi_t(t,x,c)+&\mathcal{L}^\gamma(t,x,c,-\phi_x,-\phi_{xx})\\
=-\phi_t(t,x,c)+&\inf_{(\widt,\widx,\widc)\in\setc(t,x,c)}\mathcal{G}(\widt,\widx,\widc,-\phi_x(t,x,c),-\phi_{xx}(t,x,c)).
\end{align*}
Since $(\hatt,\hatx,\hatc)\in\setc(t,x,c)$, we have
\begin{align*}
-\phi_t(t,x,c)+&\inf_{(\widt,\widx,\widc)\in\setc(t,x,c)}\mathcal{G}(\widt,\widx,u,-\phi_x(t,x,c),-\phi_{xx}(t,x,c))\\
\leq-\phi_t(t,x,c)+&\mathcal{G}(\hatt,\hatx,\hatc,-\phi_x(t,x,c),-\phi_{xx}(t,x,c)).
\end{align*}
From \eqref{equation1}, the equation above becomes
\begin{align*}
-\varphi_t(\hatt,\hatx,\hatc)+&\mathcal{G}(\hatt,\hatx,\hatc,-\varphi_x(\hatt,\hatx,\hatc),-\varphi_{xx}(\hatt,\hatx,\hatc)).
\end{align*}
which is nonpositive by \eqref{equation2}. And $-\phi_c\leq 0$ is obvious.
This proves that $v^\gamma$ is a viscosity subsolution of \eqref{subHJB}. In a similar way, we can prove that $v_\gamma$ is a viscosity supersolution of \eqref{supHJB}.
\end{proof}

\begin{lemma}\label{Alexandrov}
(Alexandrov's Theorem \cite{aleksandorov1939almost}) Let $Q\subseteq\R^n$ be a convex body and $\phi:Q\rightarrow \R$ a semiconvex (or semiconcave) function on $Q$. Then there exists a set $M\subset Q$ with the Lebesgue $|M|=0$ such that $\phi$ is twice differerntiable at any $x\in Q\setminus M$, i.e., there are $(p,P)\in\R^n\times S^n$ depending on $x$ such that
\begin{align*}
\phi(x+y)=\phi(x)+\langle p,y \rangle+\frac{1}{2}\langle Py,y \rangle+o(|y|^{2}),
~~\mbox{as $|y|\to 0$,}
\end{align*} 
where $S^n$ represent the set of all $(n\times n)$ symmetric matrices. 
\end{lemma}
\begin{lemma}\label{Jensen}
(Jensen's Lemma \cite{jensen1988the}) Let $Q\subseteq\R^n$ be a convex open set and $\phi:Q\rightarrow \R$ a semiconvex function, and $\bar{x}\in Q$ a local strict maximum of $\phi$. Then, for any constants $r,\delta>0$, the set
\begin{align*}
\mathcal{K}(\bar{x})=\big\{x\in\R^n~\big|~ &|x-\bar{x}|\leq r, ~
\mbox{there exists some $p\in\R^n$ with $0< |p|< \delta$ such that }\\
&\mbox{the function $\phi_p(\cdot)\equiv\phi(\cdot)+\langle p,\cdot\rangle$ attains a strict local maximum at $x$} \big\}
\end{align*}
has a positive Lebesgue measure. 
\end{lemma}

Now we can prove the main theorem in this section.
\begin{thm}[Comparison theorem]\label{comparison}
Assume Assumptions \ref{assumption} and \ref{assumption2} hold. Suppose $v$ is a viscosity subsolution and $\hatv$ is a viscosity supersolution of the HJB equation \eqref{HJB}, both satisfying the standard estimates \eqref{estimate1}-\eqref{estimate3}. Then $v(t,x,c)\leq\hatv(t,x,c)$ for all $(t,x,c)\in\setq$. 
\end{thm} 
 \begin{proof}
By Assumptions \ref{assumption} and \ref{assumption2}, there exists   constants $L>0$ and $\kappa\in (0,1]$ such that 
\begin{align} 
|\phi(t,x,c)-\phi(\hatt,\hatx,\hatc)|&\leq L\big[(1+|x|+|\hatx|)\big(|t-\hatt|^{\kappa}+|c-\hatc|^\kappa\big)+|x-\hatx|\big]\label{estimate6} 
\end{align}
holds for all $(t,x,c)$, $(\hatt,\hatx,\hatc)\in\setq$ and $\phi=b$, $f$, $\sigma$.
From now on, we fix $L$ and $\kappa$. 

We claim that 
\begin{align}\label{assumption of uniqueness}
v(t,x,c)\leq \hatv(t,x,c),~~\forall (t,x,c)\in [(T-\rho)^{+},T]\times\rn\times(\unc,\barc],
\end{align}
where $\rho=(2L^2+4L+1)^{-1} \in(0,1)$ with $L$ being given in \eqref{estimate6}.

Without loss of generality, we can assume $\rho<T$; otherwise we just need to perform the below argument on $[0,T]$. 
Repeating the same argument on $[(T-2\rho)^{+}, T-\rho]\times\rn\times(\unc,\barc]$ so on and so forth, we can proof the claim.

Let $\gamma\in(0,1)$ be sufficiently small, and $v^\gamma$ and $\hatv_\gamma$ be the approximations of $v$ and $\hatv$, defined in \eqref{defsemiconvex} and \eqref{defsemiconcave} respectively. 
Denote $\setg=(T-\rho,T]\times\rn\times(\unc,\barc]$ and $\mu=1+\frac{\rho}{T}$. Also, let $\alpha$, $\beta$, $\epsilon$, $\delta$ in $(0,1)$ and $\lambda\in (0,\rho)$ be sufficiently small constants that will be sent to zero eventually.
Our below estimates, the implied constants $K$ may vary from line to line, but, unless otherwise stated, they will not depend on $L$, $\alpha$, $\beta$, $\epsilon$, $\delta$, $\gamma$ or $\lambda$.

For any $(t,x,c)$, $(s,y,d)\in\setg$, define two functions 
\begin{align*}
\psi(t,s,x,y,c,d)&=\alpha \frac{2T+2\rho-t-s}{2T+2\rho} (|x|^2+|y|^2)-\beta(t+s+c+d)\\
&\quad\;+\frac{1}{2\epsilon}(|t-s|^2+|c-d|^2)+\frac{1}{2\delta}|x-y|^2++\frac{\lambda}{t-T+\rho}+\frac{\lambda}{s-T+\rho}\\
&\quad\;+\frac{\lambda}{c-\unc}+\frac{\lambda}{d-\unc},
\end{align*}
and
\begin{align}\label{defofPsi}
\Psi(t,s,x,y,c,d)&=v^\gamma(t,x,c)-\hatv_\gamma(s,y,d)-\psi(t,s,x,y,c,d).
\end{align}
Dropping all the nonnegative terms in the definition of $\psi$, we obtain  
\begin{align}
\Psi(t,s,x,y,c,d)&\leq v^\gamma(t,x,c)-\hatv_\gamma(s,y,d)+\beta(t+s+c+d)\nn\\
&\leq v^\gamma(t,x,c)-\hatv_\gamma(s,y,d)+2\beta(T+\barc).\label{inequality1}
\end{align}

Using $2T+2\rho-t-s\geq 2\rho>0$, \eqref{vupgammabound} and \eqref{vgammadownbound}, it is easy to verify 
\begin{align*}
\lim_{|x|+|y|\rightarrow\infty}\Psi(t,s,x,y,c,d)&=-\infty,\\
\lim_{t\land s\searrow T-\rho}\Psi(t,s,x,y,c,d)&=-\infty,\\
\lim_{c\land d\searrow \unc}\Psi(t,s,x,y,c,d)&=-\infty.
\end{align*}
Thus, there exists a $(t_0,s_0,x_0,y_0,c_0,d_0)\in\setg\times\setg$ 
(of course depending on the parameters $\alpha$, $\beta$, $\epsilon$, $\delta$, $\gamma$, $\lambda$) such that
\begin{align*}
\Psi(t_0,s_0,x_0,y_0,c_0,d_0)&=\max_{\setg\times\setg}\Psi.
\end{align*}
In particular, 
\begin{align*}
\Psi(t_0,s_0,x_0,y_0,c_0,d_0) 
&\geq \Psi(T,T,0,0,\barc,\barc)\\
&=v^\gamma(T,0,\barc)-\hatv_\gamma(T,0,\barc)+2\beta(T+\barc)-\frac{2\lambda}{\rho}-\frac{2\lambda}{\barc-\unc}. 
\end{align*}
After rearrangement, the above yields 
\begin{align*}
&\quad\;\alpha(|x_0|^2+|y_0|^2)+\frac{1}{2\epsilon}(|t_0-s_0|^2+|c_0-d_0|^2)+\frac{1}{2\delta}|x_0-y_0|^2\\
&\quad\quad+\frac{\lambda}{t_0-T+\rho}+\frac{\lambda}{s_0-T+\rho}+\frac{\lambda}{c_0-\unc}+\frac{\lambda}{d_0-\unc}\\
&\leq v^\gamma(t_0,x_0,c_0)-v^\gamma(T,0,\barc)-\hatv_\gamma(s_0,y_0,d_0)+\hatv_\gamma(T,0,\barc)\\
&\quad\;+\beta(t_0+s_0-2T+c_0+d_0-2\barc)+\frac{2\lambda}{\rho}+\frac{2\lambda}{\barc-\unc}\\
&\leq v^\gamma(t_0,x_0,c_0)-v^\gamma(T,0,\barc)-\hatv_\gamma(s_0,y_0,d_0)+\hatv_\gamma(T,0,\barc)+2+\frac{2}{\barc-\unc},
\end{align*}
thanks to $0<\lambda<\rho< 1$.
This together with \eqref{vupgammabound} and \eqref{vgammadownbound} leads to 
\begin{align*}\label{Penaltyfunctionbound1}
&\quad\;\alpha(|x_0|^2+|y_0|^2)+\frac{1}{2\epsilon}(|t_0-s_0|^2+|c_0-d_0|^2)+\frac{1}{2\delta}|x_0-y_0|^2\\
&\quad\;+\frac{\lambda}{t_0-T+\rho}+\frac{\lambda}{s_0-T+\rho}+\frac{\lambda}{c_0-\unc}+\frac{\lambda}{d_0-\unc}\\
&\leq K (1+|x_0|+|y_0|).
\end{align*} 
As a consequence, 
\begin{align}
&\quad\;|x_0|+|y_0|+\frac{1}{2\epsilon}(|t_0-s_0|^2+|c_0-d_0|^2)+\frac{1}{2\delta}|x_0-y_0|^2\nn\\
&\quad\;+\frac{\lambda}{t_0-T+\rho}+\frac{\lambda}{s_0-T+\rho}+\frac{\lambda}{c_0-\unc}+\frac{\lambda}{d_0-\unc}\nn\\
&\leq|x_0|+|y_0|+K(1+|x_0|+|y_0|)-\alpha(|x_0|^2+|y_0|)^2\nn\\
&\leq K_\alpha:=K+\frac{(K+1)^{2}}{2\alpha}, \label{Ka bound}
\end{align}
where the last inequality is due to the elementary inequality 
$$(K+1)x-\alpha x^{2}\leq \frac{(K+1)^{2}}{4\alpha},~~\forall\;x\in\R.$$
The above implies that 
\begin{align}
\label{limits01}|x_0|,~|y_0|&\leq K_\alpha,\\
\label{limits02} t_0, ~s_0&\geq T-\rho+\frac{\lambda}{K_\alpha},\\
\label{limits03} c_0,~d_0&\geq \unc+\frac{\lambda}{K_\alpha}.
\end{align}
So
\begin{align}\label{limits1}
x_0,y_{0}\to\bar{x}_0, ~~c_0,d_{0}\to\bar{c}_0,~~t_0,s_{0}\to\bar{t}_0,
~~\mbox{as~~$\epsilon+\delta\to 0$,}
\end{align}
with some 
$$(\bart_0,\bar{x}_0,\barc_0)\in\seta:=\Big[T-\rho+\frac{\lambda}{2K_\alpha},T\Big]\times\big[-2K_\alpha,2K_\alpha\big]\times\Big[\unc+\frac{\lambda}{2K_\alpha},\barc\Big].$$ 
Without specified illustration, all the limits taken in this proof are along some convergence subsequences. Since all the sequences considered in this proof are contained in the compact $\seta$, such convergence subsequences always exist.

We first assume, there exists a subsequence of $\gamma\to 0$, along which at least one of $\bar{t}_0=T$ and $\bar{c}_0=\barc$ holds.

Fix an arbitrary point $(t,x,c)\in\setg$. Then by \eqref{inequality1}, 
\begin{align*}
&\quad\;v^\gamma(t,x,c)-\hatv_\gamma(t,x,c)-\alpha\frac{T+\rho-t}{T+\rho}|x|^2+2\beta (t+c)+\frac{2\lambda}{t-T+\rho}+\frac{2\lambda}{c-\unc}\\
&=\Psi(t,t,x,x,c,c)\\
&\leq\Psi(t_0,s_0,x_0,y_0,c_0,d_0)\\
&\leq v^\gamma(t_0,x_0,c_0)-\hatv_\gamma(s_0,y_0,d_0)+2\beta(T+\barc).
\end{align*}
Sending $\epsilon+\delta\to 0$ in above and using \eqref{limits1}, we get
\begin{align*}
&\quad\;v^\gamma(t,x,c)-\hatv_\gamma(t,x,c)-\alpha\frac{T+\rho-t}{T+\rho}|x|^2+2\beta (t+c)+\frac{2\lambda}{t-T+\rho}+\frac{2\lambda}{c-\unc}\\
&\leq v^\gamma(\bar{t}_0,\barx_0,\barc_0)-\hatv_\gamma(\bar{t}_0,\barx_0,\barc_0)+2\beta (T+\barc)\\
&\leq v(\bar{t}_0,\barx_0,\barc_0)+K\gamma^\frac{2\kappa'}{2-\kappa'}(1+K_\alpha)^\frac{2}{2-\kappa'}-\hatv(\bar{t}_0,\barx_0,\barc_0)+ 2\beta (T+\barc),
\end{align*}
where we used \eqref{vvgammadistance}, $|\barx_0|\leq K_\alpha$ and the trivial estimate $\hatv\geq \hatv_\gamma$ to get the last inequality. 
Since either $\bar{t}_0=T$ or $\bar{c}_0=\barc$ holds, we have $v(\bar{t}_0,\barx_0,\barc_0)\leq\hatv(\bar{t}_0,\barx_0,\barc_0)$ by the boundary conditions of viscosity sub- and super-solutions.  Therefore, 
\begin{align*}
&\quad\;v^\gamma(t,x,c)-\hatv_\gamma(t,x,c)-\alpha\frac{T+\rho-t}{T+\rho}|x|^2+2\beta (t+c)+\frac{2\lambda}{t-T+\rho}+\frac{2\lambda}{c-\unc}\\
&\leq K\gamma^\frac{2\kappa'}{2-\kappa'}(1+K_\alpha)^\frac{2}{2-\kappa'}+ 2\beta (T+\barc).
\end{align*} 
Sending $\beta+\lambda+\gamma\to0$ in above, it follows that 
\begin{align*}
\;v(t,x,c)-\hatv(t,x,c)-\alpha\frac{T+\rho-t}{T+\rho}|x|^2\leq 0.
\end{align*}
Finally sending $\alpha\to 0$ yields the sired result \eqref{assumption of uniqueness}.

We now consider the remainder case: both $\bar{t}_0<T$ and $\barc_0<\barc$ hold as $\gamma\to 0$. 

In this case, 
restricting on $\seta$, the function $\psi$ is smooth and has bounded derivatives, implying its semiconcavity. Consequently, $\Psi$ is semiconvex on $\seta$, which implies, for any fixed sufficiently small constant $\theta>0$, the function 
\begin{align*}
\hat{\Psi}(t,s,x,y,c,d) &\triangleq \Psi(t,s,x,y,c,d)\\
-\theta(|t-t_0|^2&+|s-s_0|^2+|x-x_0|^2+|y-y_0|^2+|c-c_0|^2+|d-d_0|^2)
\end{align*}
is also semiconvex on $\seta$. Clearly, the unique maximizer of $\hat{\Psi}$ on $\seta$ is $(t_0,s_0,x_0,y_0,c_0,d_0)$ since $\Psi$ is maximized at $(t_0,s_0,x_0,y_0,c_0,d_0)$. 
By Lemmas \ref{Alexandrov} and \ref{Jensen}, there exist $q,\hatq, p,\hatp, r$ and $\hatr$ with 
\begin{align}\label{pqrbound}
|q|+|\hatq|+|p|+|\hatp|+|r|+|\hatr|\leq \theta,
\end{align}
and $(\hatt_0,\hats_0,\hatx_0,\haty_0,\hatc_0,\hatd_0)$ in the interior of $\seta$ with
\begin{align}\label{circlebound}
|\hatt_0-t_0|+|\hats_0-s_0|+|\hatx_0-x_0|+|\haty_0-y_0|+|\hatc_0-c_0|+|\hatd_0-d_0|\leq \theta,
\end{align}
such that the function 
\begin{align*}
\hat{\Psi}&(t,s,x,y,c,d)+qt+\hatq s+px+\hatp y+rc+\hatr d 
\end{align*}
attains its maximum at $(\hatt_0,\hats_0,\hatx_0,\haty_0,\hatc_0,\hatd_0)$ on $\seta$, at which point $v^\gamma(t,x,c)-v_\gamma(s,y,d)$ is twice differentiable. 

For notational simplicity, we now drop $\gamma$ in $v^\gamma(t,x,c)$ and $\hatv_\gamma(s,y,d)$. Then, by the first- and second-order neccessary conditions for a maximum point, at the point $(\hatt_0,\hats_0,\hatx_0,\haty_0,\hatc_0,\hatd_0)$, we must have
\begin{align}
&v_t=\psi_t+2\theta(\hatt_0-t_0)-q,\label{vt}\\
&\hatv_s=-\psi_s-2\theta(\hat{s_0}-s_0)+\hatq,\nn\\
&v_x=\psi_x+2\theta(\hatx_0-x_0)-p,\nn\\
&\hatv_y=-\psi_y-2\theta(\haty_0-y_0)+\hatp,\nn\\
&v_c=\psi_c+2\theta(\hatc_0-c_0)-r,\nn\\
&\hatv_d=-\psi_d-2\theta(\hatd_0-d_0)+\hatr,\nn\\
&\left(
\begin{array}{cc}
v_{xx} & 0 \\
0 & -\hatv_{yy}
\end{array}
\right)
\leq
\left(
\begin{array}{cc}
\psi_{xx} & \psi_{xy} \\
\psi_{xy} & \psi_{yy}
\end{array}
\right)
+2\theta \mathbf{I}_{2N}.\nn
\end{align}
where $\mathbf{I}_{2N}$ is the $2N\times2N$ identity matrix. Now, at $(\hatt_0,\hats_0,\hatx_0,\haty_0,\hatc_0,\hatd_0)$, we have the following:
\begin{align*}
\psi_t &=-\beta-\frac{\lambda}{(\hatt_0-T+\rho)^2}-\frac{\alpha}{2T+2\rho}(|\hatx_0|^2+|\haty_0|^2)+\frac{1}{\epsilon}(\hatt_0-\hats_0),\nn\\
\psi_s &=-\beta-\frac{\lambda}{(\hats_0-T+\rho)^2}-\frac{\alpha}{2T+2\rho}(|\hatx_0|^2+|\haty_0|^2)+\frac{1}{\epsilon}(\hats_0-\hatt_0),\nn\\
\psi_x&=\alpha\frac{2T+2\rho-\hatt_0-\hats_0}{T+\rho}\hatx_0+\frac{\hatx_0-\haty_0}{\delta},\nn\\
\psi_y&=\alpha\frac{2T+2\rho-\hatt_0-\hats_0}{T+\rho}\haty_0+\frac{\haty_0-\hatx_0}{\delta},\nn\\
\psi_c &=\frac{\hatc_0-\hatd_0}{\epsilon}-\beta-\frac{\lambda}{(\hatc_0-\unc)^2},\nn\\
\psi_d &=\frac{\hatd_0-\hatc_0}{\epsilon}-\beta-\frac{\lambda}{(\hatd_0-\unc)^2},
\end{align*}
and 
\begin{align}
\begin{pmatrix}
\psi_{xx} & \psi_{xy} \\
\psi_{xy} & \psi_{yy}
\end{pmatrix}
=\frac{1}{\delta}
\begin{pmatrix}
I_N & -I_N \\
-I_N & I_N
\end{pmatrix}
+\alpha\frac{2T+2\rho-\hatt_0-\hats_0}{T+\rho}\mathbf{I}_{2N}.\label{defofA}
\end{align}

Combining equations from \eqref{vt} to \eqref{defofA}, we know
\begin{align}
&v_t=-\beta-\frac{\lambda}{(\hatt_0-T+\rho)^2}-\frac{\alpha}{2T+2\rho}(|\hatx_0|^2+|\haty_0|^2)+\frac{1}{\epsilon}(\hatt_0-\hats_0)+2\theta(\hatt_0-t_0)-q,\label{vtbound}\\
&\hatv_s=\beta+\frac{\lambda}{(\hats_0-T+\rho)^2}+\frac{\alpha}{2T+2\rho}(|\hatx_0|^2+|\haty_0|^2)-\frac{1}{\epsilon}(\hats_0-\hatt_0)-2\theta(\hat{s_0}-s_0)+\hatq,\label{hatvsbound}\\
&v_x=\alpha\frac{2T+2\rho-\hatt_0-\hats_0}{T+\rho}\hatx_0+\frac{\hatx_0-\haty_0}{\delta}+2\theta(\hatx_0-x_0)-p,\label{vxbound}\\
&\hatv_y=-\alpha\frac{2T+2\rho-\hatt_0-\hats_0}{T+\rho}\haty_0-\frac{\haty_0-\hatx_0}{\delta}-2\theta(\haty_0-y_0)+\hatp,\label{hatvybound}\\
&v_c=\frac{\hatc_0-\hatd_0}{\epsilon}-\beta-\frac{\lambda}{(\hatc_0-\unc)^2}+2\theta(\hatc_0-c_0)-r,\label{vcbound}\\
&\hatv_d=-\frac{\hatd_0-\hatc_0}{\epsilon}+\beta+\frac{\lambda}{(\hatd_0-\unc)^2}-2\theta(\hatd_0-d_0)+\hatr\label{hatvdbound}\\
&A\equiv\left(
\begin{array}{cc}
v_{xx} & 0 \\
0 & -\hatv_{yy}
\end{array}
\right)
\leq
\frac{1}{\delta}
\left(
\begin{array}{cc}
I_N & -I_N \\
-I_N & I_N
\end{array}
\right)+\alpha\frac{2T+2\rho-\hatt_0-\hats_0}{T+\rho}\mathbf{I}_{2N}
+2\theta \mathbf{I}_{2N}.\label{vxxvyy}
\end{align} 

We first assume 
\begin{align}\label{case1}
-\hatv_d(\hats_0,\haty_0,\hatd_0)\geq 0.
\end{align}
By Lemma \ref{gammasupersubviscosity} and the definition of viscosity subsolution, we have 
\begin{align*}
-v_c(\hatt_0,\hatx_0,\hatc_0)\leq0. 
\end{align*}
Together with 
\eqref{vcbound} and \eqref{hatvdbound}, we get
\begin{align*}
&\quad\;v_c(\hatt_0,\hatx_0,\hatc_0)-\hatv_d(\hats_0,\haty_0,\hatd_0)\\
&=-2\beta-\frac{\lambda}{(\hatc_0-\unc)^2}-\frac{\lambda}{(\hatd_0-\unc)^2}+2\theta(\hatc_0-c_0)-r+2\theta(\hatd_0-d_0)+\hatr\geq 0,
\end{align*}
so
\begin{align*}
-2\beta+2\theta(\hatc_0-c_0)-r+2\theta(\hatd_0-d_0)+\hatr\geq \frac{\lambda}{(\hatc_0-\unc)^2}+\frac{\lambda}{(\hatd_0-\unc)^2}>0.
\end{align*}
Sending $\theta\to0$ in above and using \eqref{pqrbound} and \eqref{circlebound}, we get $-2\beta\geq 0$, contradicting to $\beta>0$. 

The above contradiction shows that \eqref{case1} cannot hold. Therefore, by the definition of viscosity supersolution, we must have 
\begin{align}\label{vsviscosity}
&-\hatv_s+\mathcal{L}_\gamma (\hats_0,\haty_0,\hatd_0)\geq 0.
\end{align}
Here and hereafter, for notation simplicity, we write $-\hatv_s+\mathcal{L}_\gamma (\hats_0,\haty_0,\hatd_0)$ instead of 
\[-\hatv_s(s,y,d)+\mathcal{L}_\gamma (s,y,d,-\hatv_x(s,y,d),-\hatv_{xx}(s,y,d))\big|_{\hats_0,\haty_0,\hatd_0}.\]
and similar for $\mathcal{L}^{\gamma}$.

Recalling the definition of $\setc(t,x,c)$ in \eqref{defupgamma} and \eqref{defdowngamma}, for any $(t,x,c)\in\setc(\hatt_0,\hatx_0,\hatc_0)$, it should satisfy the following
\begin{align}\label{circlebound with alpha}
|t-\hatt_0|^2+|x-\hatx_0|^2+|c-\hatc_0|^2
&\leq 2\gamma^2 K\gamma^\frac{2\kappa'}{2-\kappa'}(1+K_\alpha)^\frac{2}{2-\kappa'}\nn\\
&= 2K\gamma^\frac{4}{2-\kappa'}(1+K_\alpha)^\frac{2}{2-\kappa'}.
\end{align} 
The above inequality means that $(t,x,c)\in\setw$ where 
\begin{align}\label{defofsetw}
\setw\triangleq[\hatt_0-K_{\alpha\gamma},\hatt_0+K_{\alpha\gamma}]\times[\hatx_0-K_{\alpha\gamma},\hatx_0+K_{\alpha\gamma}]\times[\hatc_0-K_{\alpha\gamma},\hatc_0+K_{\alpha\gamma}],
\end{align}
and 
\begin{align}\label{defofkalphagamma}
K_{\alpha\gamma}\triangleq \sqrt{2K}\gamma^\frac{2}{2-\kappa'} (1+K_\alpha)^{\frac{1}{2-\kappa'}}.
\end{align}
So $\setc(\hatt_0,\hatx_0,\hatc_0)$ can be rewritten as follow:
\begin{align*}
\setc(\hatt_0,\hatx_0,\hatc_0)=\Big\{(t,x,c)\in\setw\Big|~\frac{1}{2\gamma^2}[|t-\hatt_0|^2+|x-\hatx_0|^2+|c-\hatc_0|^2]\leq K\gamma^\frac{2\kappa'}{2-\kappa'}(1+K_\alpha)^\frac{2}{2-\kappa'} \Big\}.
\end{align*}
From Assumptions \ref{assumption} and \ref{assumption2}, we know that $b$, $f$ and $\sigma$ are continuous. So we get the continuity of the below function in $(\wids,\widy,\widd)$: 
\begin{align*}
&-\frac{1}{2}\hatv_{yy}(\hats_0,\haty_0,\hatd_0)\sigma^2(\wids,\widy,\widd)-\hatv_y(\hats_0,\haty_0,\hatd_0) b(\wids,\widy,\widd)-f(\wids,\widy,\widd). 
\end{align*}
Since $\setw$ is a closed set, by combining the above, we conclude there exists a $(s^*,y^*,d^*)\in\setc(\hats_0,\haty_0,\hatd_0)$ such that 
\begin{align}
\quad\mathcal{L}_\gamma(\hats_0,\haty_0,\hatd_0)
&=\inf_{(\wids,\widy,\widd)\in\setc(\hats_0,\haty_0,\hatd_0)}\mathcal{G}(\wids,\widy,\widd,-\hatv_x(\hats_0,\haty_0,\hatd_0),-\hatv_{xx}(\hats_0,\haty_0,\hatd_0))\nn\\
&=\mathcal{G}(s^*,y^*,d^*,-\hatv_x(\hats_0,\haty_0,\hatd_0),-\hatv_{xx}(\hats_0,\haty_0,\hatd_0))\nn\\ 
&=-\frac{1}{2}\hatv_{yy}(\hats_0,\haty_0,\hatd_0)\sigma^2(s^*,y^*,d^*)-\hatv_y(\hats_0,\haty_0,\hatd_0) b(s^*,y^*,d^*)\nn\\
&\quad\;-f(s^*,y^*,d^*).\label{ldowngamma} 
\end{align}
Similarly, we know that there exists a $(t^*,x^*,c^*)\in\setc(\hatt_0,\hatx_0,\hatc_0)$, such that 
\begin{align}
\mathcal{L}^\gamma(\hatt_0,\hatx_0,\hatc_0)=&-\frac{1}{2}v_{xx}(\hatt_0,\hatx_0,\hatc_0)\sigma^2(t^*,x^*,c^*)-v_x(\hatt_0,\hatx_0,\hatc_0) b(t^*,x^*,c^*)\nn\\
&-f(t^*,x^*,c^*). \label{lupgamma}
\end{align}
On the other hand, by Lemma \ref{gammasupersubviscosity} and the definition of viscosity subsolution, it holds that 
\begin{align}\label{vtviscosity}
-v_t+\mathcal{L}^\gamma (\hatt_0,\hatx_0,\hatc_0)\leq0.
\end{align}
Together with \eqref{vsviscosity}, we get
\begin{align*}
\hatv_s(\hats_0,\haty_0,\hatd_0)-v_t(\hatt_0,\hatx_0,\hatc_0)
&\leq\mathcal{L}_\gamma (\hats_0,\haty_0,\hatd_0)-\mathcal{L}^\gamma(\hatt_0,\hatx_0,\hatc_0).
\end{align*}
Using \eqref{ldowngamma} and \eqref{lupgamma}, the above leads to
\begin{align}
&\quad\; \hatv_s(\hats_0,\haty_0,\hatd_0)-v_t(\hatt_0,\hatx_0,\hatc_0)\nn\\
&\leq\Big\{-\frac{1}{2}\hatv_{yy}(\hats_0,\haty_0,\hatd_0)\sigma^2(s^*,y^*,d^*)-\hatv_y(\hats_0,\haty_0,\hatd_0)b(s^*,y^*,d^*)-f(s^*,y^*,d^*)\Big\}\nn\\
&\quad\;-\Big\{-\frac{1}{2}v_{xx}(\hatt_0,\hatx_0,\hatc_0)\sigma^2(t^*,x^*,c^*) -v_x(\hatt_0,\hatx_0,\hatc_0)b(t^*,x^*,c^*) -f(t^*,x^*,c^*)\Big\}\nn\\
& = (I)+(II)+(III),\label{equation5}
\end{align}
where 
\begin{align*}
(I)&\triangleq \frac{1}{2}v_{xx}(\hatt_0,\hatx_0,\hatc_0)\sigma^2(t^*,x^*,c^*)-\frac{1}{2}\hatv_{yy}(\hats_0,\haty_0,\hatd_0)\sigma^2(s^*,y^*,d^*),\\
(II)&\triangleq v_x(\hatt_0,\hatx_0,\hatc_0)b(t^*,x^*,c^*)-\hatv_y(\hats_0,\haty_0,\hatd_0)b(s^*,y^*,d^*),\\
(III)&\triangleq f(t^*,x^*, c^*)-f(s^*,y^*,d^*).
\end{align*} 
We first estimate the left-hand side of \eqref{equation5}.
By \eqref{vtbound} and \eqref{hatvsbound}, we have
\begin{align}
&\quad\; \hatv_s(\hats_0,\haty_0,\hatd_0)-v_t(\hatt_0,\hatx_0,\hatc_0)\nn\\
&=2\beta+\frac{\alpha}{T+\rho}(|\hatx_0|^2+|\haty_0|^2)+\frac{\lambda}{(\hatt_0-T+\rho)^2}\nn\\
&\quad\;+\frac{\lambda}{(\hats_0-T+\rho)^2}-2\theta(\hatt_0-t_0+\hats_0-s_0)+q+\hatq\nn\\
& \geq 2\beta+\frac{\alpha}{T+\rho}(|\hatx_0|^2+|\haty_0|^2)-2\theta(\hatt_0-t_0+\hats_0-s_0)+q+\hatq.\label{estimateofLHS}
\end{align} 
We now estimate the right-hand side of \eqref{equation5}. 
The below estimates hold for $\phi=b$, $f$, $\sigma$. 
First, by \eqref{estimate6}, we have 
\begin{align*}
&\quad\;|\phi(t^*,x^*, c^*)-\phi(s^*,y^*,d^*)|\nn\\
&\leq L\big[(1+|x^*|+|y^*|)|t^*-s^*|^{\kappa}+|x^*-y^*|+(1+|x^*|+|y^*|)| c^*-d^*|^\kappa\big].
\end{align*}
Furthermore, by the definition of $\setw$ in \eqref{defofsetw}, we get
\begin{align*}
|t^*-s^*|\leq|t^*-\hatt_0|+|\hatt_0-\hats_0|+|\hats_0-s^*|\leq 2K_{\alpha\gamma}+|\hatt_0-\hats_0|,
\end{align*}
and similarly,
\begin{align*}
|x^*-y^*|\leq2K_{\alpha\gamma}+|\hatx_0-\haty_0|,~~
|c^*-d^*|\leq2K_{\alpha\gamma}+|\hatc_0-\hatd_0|.
\end{align*} 
By the definition of $\seta$, we have $|\hatx_0|$, $|\haty_0|\leq 2K_\alpha$. 
Combining the above yields
\begin{align}\label{estimateofphi1}
&\quad\;|\phi(t^*,x^*,c^*)-\phi(s^*,y^*,d^*)|\nn\\
&\leq L\big[(1+|\hatx_0|+|\haty_0|+2K_{\alpha\gamma})(2K_{\alpha\gamma}+|\hatt_0-\hats_0|)^{\kappa}+2K_{\alpha\gamma}+|\hatx_0-\haty_0|\nn\\
&\quad\;+(1+|\hatx_0|+|\haty_0|+2K_{\alpha\gamma})(2K_{\alpha\gamma}+|\hatc_0-\hatd_0|)^{\kappa}\big]\nn\\
&\leq L\big[(1+4K_\alpha+2K_{\alpha\gamma})(2K_{\alpha\gamma}+|\hatt_0-\hats_0|)^{\kappa}+2K_{\alpha\gamma}+|\hatx_0-\haty_0|\nn\\
&\quad\;+(1+4K_\alpha+2K_{\alpha\gamma})(2K_{\alpha\gamma}+|\hatc_0-\hatd_0|)^{\kappa}\big].
\end{align}
Moreover, by Assumption \ref{assumption} and \eqref{defofsetw}, we have 
\begin{align*}
|\phi(t^*,x^*, c^*)| \leq L(1+|x^*|)\leq L(1+|\hatx_0|+|x^*-\hatx_0|) \leq L(1+|\hatx_0|+K_{\alpha\gamma}).
\end{align*}
Similarly, we get
\begin{align}\label{estimateofphi2}
|\phi(s^*,y^*,u^*)|\leq L(1+|\haty_0|+K_{\alpha\gamma}). 
\end{align}
Next, we deal with $(I)$, $(II)$ and $(III)$. 
By \eqref{vxxvyy},
\begin{align}
(I)&=\frac{1}{2}v_{xx}(\hatt_0,\hatx_0,\hatc_0)\sigma^2(t^*,x^*,c^*)-\frac{1}{2}\hatv_{yy}(\hats_0,\haty_0,\hatd_0)\sigma^2(s^*,y^*,d^*)\nn\\
&\leq\frac{1}{2}\trace{\left(
\begin{array}{cc}
\sigma(t^*,x^*,c^*) \\
\sigma(s^*,y^*,d^*)
\end{array}
\right)^\top A \left(
\begin{array}{cc}
\sigma(t^*,x^*,c^*) \\
\sigma(s^*,y^*,d^*)
\end{array}
\right)}\nn\\
&\leq\frac{1}{2}\Big\{\frac{1}{\delta}|\sigma(t^*,x^*,c^*)-\sigma(s^*,y^*,d^*)|^2\nn\\
&\quad\;+\alpha\frac{2T+2\rho-\hatt_0-\hats_0}{T+\rho}\big(|\sigma(t^*,x^*, c^*)|^2+|\sigma(s^*,y^*,d^*)|^2\big)\Big\}.
\end{align}
By \eqref{estimateofphi1} and \eqref{estimateofphi2}, 
\begin{align}
(I)&\leq \frac{L^2}{2\delta}\big[(1+4K_\alpha+2K_{\alpha\gamma})(2K_{\alpha\gamma}+|\hatt_0-\hats_0|)^{\kappa}+2K_{\alpha\gamma}+|\hatx_0-\haty_0|\nn\\
&\quad\;+(1+4K_\alpha+2K_{\alpha\gamma})(2K_{\alpha\gamma}+|\hatc_0-\hatd_0|)^{\kappa}\big]^2\nn\\
&\quad\;+\alpha\frac{2T+2\rho-\hatt_0-\hats_0}{2T+2\rho}\big[ L^2(1+|\hatx_0|+K_{\alpha\gamma})^2+ L^2(1+|\haty_0|+K_{\alpha\gamma})^2\big].\label{estimateof1} 
\end{align}
By \eqref{vxbound} and \eqref{hatvybound}, we have 
\begin{align*}
(II)&=\Big[\alpha\frac{2T+2\rho-\hatt_0-\hats_0}{T+\rho}\hatx_0+\frac{\hatx_0-\haty_0}{\delta}+2\theta(\hatx_0-x_0)-p\Big]b(t^*,x^*,c^*)\nn\\
&\quad\;+\Big[\alpha\frac{2T+2\rho-\hatt_0-\hats_0}{T+\rho}\haty_0+\frac{\haty_0-\hatx_0}{\delta}+2\theta(\haty_0-y_0)-p\Big]b(s^*,y^*,d^*)\nn\\
&=\Big[\alpha\frac{2T+2\rho-\hatt_0-\hats_0}{T+\rho}\hatx_0+2\theta(\hatx_0-x_0)-p\Big]b(t^*,x^*,c^*)\nn\\
&\quad\;+\Big[\alpha\frac{2T+2\rho-\hatt_0-\hats_0}{T+\rho}\haty_0+2\theta(\haty_0-y_0)-p\Big]b(s^*,y^*,d^*)\nn\\
&\quad\;+\frac{\hatx_0-\haty_0}{\delta}(b(t^*,x^*,c^*)-b(s^*,y^*,d^*)).\nn
\end{align*}
By \eqref{pqrbound}, \eqref{circlebound} and above, 
\begin{align*}
(II)&\leq\Big(\alpha\frac{2T+2\rho-\hatt_0-\hats_0}{T+\rho}|\hatx_0|+2\theta^2+\theta\Big)|b(t^*,x^*,c^*)|\nn\\
&\quad\;+\Big(\alpha\frac{2T+2\rho-\hatt_0-\hats_0}{T+\rho}|\haty_0|+2\theta^2+\theta\Big)|b(s^*,y^*,d^*)|\nn\\
&\quad\;+\frac{|\hatx_0-\haty_0|}{\delta}|b(t^*,x^*,c^*)-b(s^*,y^*,d^*)|.\nn 
\end{align*}
By \eqref{estimateofphi1} and \eqref{estimateofphi2}, the above estimate leads to 
\begin{align}
(II)&\leq\Big(\alpha\frac{2T+2\rho-\hatt_0-\hats_0}{T+\rho}|\hatx_0|+2\theta^2+\theta\Big)L(1+|\hatx_0|+K_{\alpha\gamma})\nn\\
&\quad\;+\Big(\alpha\frac{2T+2\rho-\hatt_0-\hats_0}{T+\rho}|\haty_0|+2\theta^2+\theta\Big)L(1+|\haty_0|+K_{\alpha\gamma})\nn\\
&\quad\;+\frac{|\hatx_0-\haty_0|}{\delta} L\big[(1+4K_\alpha+2K_{\alpha\gamma})(2K_{\alpha\gamma}+|\hatt_0-\hats_0|)^{\kappa}\nn\\
&\quad\;+2K_{\alpha\gamma}+|\hatx_0-\haty_0|+(1+4K_\alpha+2K_{\alpha\gamma})(2K_{\alpha\gamma}+|\hatc_0-\hatd_0|)^{\kappa}\big].\label{estimateof2}
\end{align}
By \eqref{estimateofphi1}, we have
\begin{align}\label{estimateof3}
(III)&=f(t^*,x^*,c^*)-f(s^*,y^*,d^*)\nn\\
&\leq L\big[(1+4K_\alpha+2K_{\alpha\gamma})(2K_{\alpha\gamma}+|\hatt_0-\hats_0|)^{\kappa}+2K_{\alpha\gamma}+|\hatx_0-\haty_0|\nn\\
&\quad\;+(1+4K_\alpha+2K_{\alpha\gamma})(2K_{\alpha\gamma}+|\hatc_0-\hatd_0|)^{\kappa}\big].
\end{align}
Combine \eqref{estimateofLHS}, \eqref{estimateof3}, \eqref{estimateof2} and \eqref{estimateof1}, \eqref{equation5} yields 
\begin{align}
&\quad\;2\beta+\frac{\alpha}{T+\rho}(|\hatx_0|^2+|\haty_0|^2)-2\theta(\hatt_0-t_0+\hats_0-s_0)+q+\hatq\nn\\
&\leq L\big[(1+4K_\alpha+2K_{\alpha\gamma})(2K_{\alpha\gamma}+|\hatt_0-\hats_0|)^{\kappa}+2K_{\alpha\gamma}+|\hatx_0-\haty_0|\nn\\
&\quad\;+(1+4K_\alpha+2K_{\alpha\gamma})(2K_{\alpha\gamma}+|\hatc_0-\hatd_0|)^{\kappa}\big]\nn\\
&\quad\;+\Big(\alpha\frac{2T+2\rho-\hatt_0-\hats_0}{T+\rho}|\hatx_0|+2\theta^2+\theta\Big)L(1+|\hatx_0|+K_{\alpha\gamma})\nn\\
&\quad\;+\Big(\alpha\frac{2T+2\rho-\hatt_0-\hats_0}{T+\rho}|\haty_0|+2\theta^2+\theta\Big)L(1+|\haty_0|+K_{\alpha\gamma})\nn\\
&\quad\;+\frac{|\hatx_0-\haty_0|}{\delta} L\big[(1+4K_\alpha+2K_{\alpha\gamma})(2K_{\alpha\gamma}+|\hatt_0-\hats_0|)^{\kappa}\nn\\
&\quad\;+2K_{\alpha\gamma}+|\hatx_0-\haty_0|+(1+4K_\alpha+2K_{\alpha\gamma})(2K_{\alpha\gamma}+|\hatc_0-\hatd_0|)^{\kappa}\big]\nn\\
&\quad\;+\frac{L^2}{2\delta}\big[(1+4K_\alpha+2K_{\alpha\gamma})(2K_{\alpha\gamma}+|\hatt_0-\hats_0|)^{\kappa}+2K_{\alpha\gamma}+|\hatx_0-\haty_0|\nn\\
&\quad\;+(1+4K_\alpha+2K_{\alpha\gamma})(2K_{\alpha\gamma}+|\hatc_0-\hatd_0|)^{\kappa}\big]^2\nn\\
&\quad\;+\alpha\frac{2T+2\rho-\hatt_0-\hats_0}{2T+2\rho}\big[ L^2(1+|\hatx_0|+K_{\alpha\gamma})^2+ L^2(1+|\haty_0|+K_{\alpha\gamma})^2\big].\label{estimateof1+2+3}
\end{align}
By \eqref{pqrbound} and \eqref{circlebound}, we have
\begin{align*}
\lim_{\theta\to0}(\hatt_0,\hats_0,\hatx_0,\haty_0,\hatc_0,\hatd_0)&=(t_0,s_0,x_0,y_0,c_0,d_0),\\
\lim_{\theta\to0}(p,q,r)&=(0,0,0).
\end{align*}
So sending $\theta\to0$ in \eqref{estimateof1+2+3} leads to
\begin{align}
&\quad\; 2\beta+\frac{\alpha}{T+\rho}(|x_0|^2+|y_0|^2)\nn\\
&\leq L\big[(1+4K_\alpha+2K_{\alpha\gamma})(2K_{\alpha\gamma}+|t_0-s_0|)^{\kappa}+2K_{\alpha\gamma}\nn\\
&\quad\;+|x_0-y_0|+(1+4K_\alpha+2K_{\alpha\gamma})(2K_{\alpha\gamma}+|c_0-d_0|)^\kappa\big]\nn\\
&\quad\;+\Big(\alpha\frac{2T+2\rho-t_0-s_0}{T+\rho}|x_0|\Big)L(1+|x_0|+K_{\alpha\gamma})\nn\\
&\quad\;+\Big(\alpha\frac{2T+2\rho-t_0-s_0}{T+\rho}|y_0|\Big)L(1+|y_0|+K_{\alpha\gamma})\nn\\
&\quad\;+\frac{|x_0-y_0|}{\delta}L\big[(1+4K_\alpha+2K_{\alpha\gamma})(2K_{\alpha\gamma}+|t_0-s_0|)^{\kappa}+2K_{\alpha\gamma}+|x_0-y_0|\nn\\
&\quad\;+(1+4K_\alpha+2K_{\alpha\gamma})(2K_{\alpha\gamma}+|c_0-d_0|)^\kappa\big]\nn\\
&\quad\;+\frac{1}{2\delta}\Big\{L\big[(1+4K_\alpha+2K_{\alpha\gamma})(2K_{\alpha\gamma}+|t_0-s_0|)^{\kappa}+2K_{\alpha\gamma}+|x_0-y_0|\nn\\
&\quad\;+(1+4K_\alpha+2K_{\alpha\gamma})(2K_{\alpha\gamma}+|c_0-d_0|)^\kappa\big]^2\nn\\
&\quad\;+\alpha\frac{2T+2\rho-t_0-s_0}{2T+2\rho}[L^2( 1+|x_0|+K_{\alpha\gamma})^2+L^2( 1+|y_0|+K_{\alpha\gamma})^2]\Big\}.\label{estimateof1+2+3(2)}
\end{align}
Suppose that 
\begin{align*}
\lim_{\epsilon+\gamma\to0}(t_0,s_0,x_0,y_0,c_0,d_0)=(t_\delta,s_\delta,x_\delta,y_\delta,c_\delta,d_\delta).
\end{align*}
Now we deal with \eqref{Ka bound} again. Sending $\epsilon+\gamma\to0$ in \eqref{Ka bound}, it leads to
\begin{align}\label{ka bound2}
&\quad\;|x_\delta|+|y_\delta|+\lim_{\epsilon+\gamma\to0}\Big[\frac{1}{2\epsilon}(|t_0-s_0|^2+|c_0-d_0|^2)\Big]+\frac{1}{2\delta}|x_\delta-y_\delta|^2\nn\\
&\quad\;+\frac{\lambda}{t_\delta-T+\rho}+\frac{\lambda}{s_\delta-T+\rho}+\frac{\lambda}{c_\delta-\unc}+\frac{\lambda}{d_\delta-\unc}\leq K_\alpha. 
\end{align}
The above implies $t_\delta=s_\delta$, $c_\delta=d_\delta$.
Note that by \eqref{defofkalphagamma}, $\lim_{\gamma\to0}K_{\alpha\gamma}=0$. It hence follows \begin{align*}
\lim_{\ep+\gamma\to0}(2K_{\alpha\gamma}+|t_0-s_0|^\kappa)=\lim_{\ep+\gamma\to0}(2K_{\alpha\gamma}+|c_0-d_0|^\kappa)=0. 
\end{align*}
Sending $\epsilon+\gamma\to0$ in \eqref{estimateof1+2+3(2)}, we get 
\begin{align}
&\quad\;2\beta+\frac{\alpha}{T+\rho}(|x_\delta|^2+|y_\delta|^2)\nn\\
&\leq L(|x_\delta-y_\delta|)+\alpha\frac{2T+2\rho-2t_\delta}{T+\rho}L\big[x_\delta (1+|x_\delta|)+y_\delta(1+|y_\delta|)\big]\nn\\
&\quad\;+\frac{L|y_\delta-x_\delta|^2}{\delta}+\frac{L|x_\delta-y_\delta|^2}{2\delta} +\alpha\frac{2T+2\rho-2t_\delta}{2T+2\rho}L^2\big[(1+|x_\delta|)^2+(1+|y_\delta|)^2\big].\label{estimateof1+2+3(3)} 
\end{align}
Suppose that
\begin{align*}
\lim_{\delta\to0}(t_\delta,t_\delta,x_\delta,y_\delta,c_\delta,c_\delta)=(\bar{t}_1,\bar{t}_1,\barx_1,\bary_1,\barc_1,\barc_1).
\end{align*} 
Sending $\delta\to0$ in \eqref{ka bound2}, we get $\barx_1=\bary_1$.

Moreover, we deal with \eqref{defofPsi} again. From the inequality
\begin{align}\label{Kalpha bound}
2\Psi(t_0,s_0,x_0,y_0,c_0,d_0)\geq \Psi(t_0,t_0,x_0,x_0,c_0,c_0)+\Psi(s_0,s_0,y_0,y_0,d_0,d_0),
\end{align}
it follows that
\begin{align*}
&\quad\; 2v^\gamma(t_0,x_0,c_0)-2\hatv_\gamma(s_0,y_0,d_0)-2\alpha\Big(1-\frac{t_0+s_0}{2T+2\rho}\Big)(|x_0|^2+|y_0|^2)\\
&\quad\;+2\beta(t_0+s_0+c_0+d_0)-\frac{1}{\epsilon}(|t_0-s_0|^2+|c_0-d_0|^2)-\frac{1}{\delta}|x_0-y_0|^2-\frac{2\lambda}{t_0-T+\rho}\\
&\quad\;-\frac{2\lambda}{s_0-T+\rho}+\frac{2\lambda}{c_0-\unc}+\frac{2\lambda}{d_0-\unc}\\
&\geq v^\gamma(t_0,x_0,c_0)-\hatv_\gamma(t_0,x_0,c_0)-2\alpha\Big(1-\frac{2t_0}{2T+2\rho}\Big)|x_0|^2+2\beta (t_0+c_0)\\
&\quad\;-\frac{2\lambda}{t_0-T+\rho}-\frac{2\lambda}{c_0-\unc}+ v^\gamma(s_0,y_0,d_0)-\hatv_\gamma(s_0,y_0,d_0)\\
&\quad\;-2\alpha\Big(1-\frac{2s_0}{2T+2\rho}\Big)|y_0|^2+2\beta(s_0+d_0)-\frac{2\lambda}{s_0-T+\rho}-\frac{2\lambda}{d_0-\unc}.
\end{align*}
After cancellation and rearrangement, the above becomes
\begin{align*}
&\quad\; \frac{1}{\epsilon}(|t_0-s_0|^2+|c_0-d_0|^2)+\frac{1}{\delta}|x_0-y_0|^2\nn\\
&\leq v^\gamma(t_0,x_0,c_0)-v^\gamma(s_0,y_0,d_0)+\hatv_\gamma(t_0,x_0,c_0)-\hatv_\gamma(s_0,y_0,d_0) \\
&\quad\;+ \frac{\alpha}{T+\rho}(s_0-t_0)(|x_0|^2-|y_0|^2).
\end{align*}
By Lemma \ref{lemma1.1} and Lemma \ref{lemma2}, noting that $(t_0,x_0,c_0)$, $(s_0, y_0,d_0)\in\seta$, we have 
\begin{align}
&\quad\; \frac{1}{\epsilon}(|t_0-s_0|^2+|c_0-d_0|^2)+\frac{1}{\delta}|x_0-y_0|^2\nn\\
&\leq 2K(1+|x_0|)^\frac{4}{3}\gamma^\frac{2}{3}+2K(1+|y_0|)^\frac{4}{3}\gamma^\frac{2}{3}\nn\\
&\quad\;+2K\big[|x_0-y_0|+(1+|x_0|\vee|y_0|)|t_0-s_0|^{\frac{1}{2}}+|c_0-d_0|\big]\nn\\
&\quad\;+\frac{\alpha}{T+\rho}(s_0-t_0)(|x_0|^2-|y_0|^2).\label{estimate4}
\end{align} 
Sending $\epsilon+\delta+\gamma\to0$ in \eqref{estimate4}, we conclude that
\begin{align*}
\lim_{\epsilon+\delta+\gamma\to0}\frac{1}{\epsilon}(|t_0-s_0|^2+|c_0-d_0|^2)+\frac{1}{\delta}|x_0-y_0|^2=0.
\end{align*} 
Sending $\delta\to0$ in \eqref{estimateof1+2+3(3)} leads to
\begin{align*}
2\beta+\frac{\alpha}{T+\rho}2|\barx_1|^2\leq \alpha\frac{2T+2\rho-2\bar{t}_1}{T+\rho}2L(|\barx_1|+|\barx_1|^2)+\alpha\frac{2T+2\rho-2\bar{t}_1}{T+\rho}L^2(1+|\barx_1|)^2,
\end{align*} 
namely, 
\begin{align*} 
\beta (T+\rho)&\leq \alpha\big\{\big[-1+(T+\rho-\bar{t}_1)(2L+L^2)\big]|\barx_1|^2\nn\\
&\qquad\quad+(T+\rho-\bar{t}_1)(2L+2L^2)|\barx_1|+(T+\rho-\bar{t}_1)L^2\big\}.
\end{align*}
Recall that $\bar{t}_1\in [T-\rho,T]$ and $\rho=\frac{1}{4(2L+L^2+1)}$, so 
\begin{align*} 
T+\rho-\bar{t}_1\leq T+\rho-(T-\rho)=2\rho\leq \frac{1}{2(2L+L^2)}.
\end{align*}
Combining the above two estimates, we get 
\begin{align*} 
\beta (T+\rho)&\leq \alpha\Big\{-\frac{1}{2}|\barx_1|^2+\frac{1}{2}|\barx_1| +\frac{1}{2}\Big\}=\frac{\alpha}{8}\big\{5- (2|\barx_1|-1)^2\big\}\leq \frac{5\alpha}{8}.
\end{align*}
Sending $\alpha\rightarrow0$, we obtain $\beta (T+\rho)\leq 0$, which contradicts our assumption $\beta>0$. And this proves the theorem.
\end{proof}
From Theorem \ref{comparison}, we can get the uniqueness of the HJB equation \eqref{HJB} in viscosity sense.

\bibliographystyle{siam}

\begin{thebibliography}{}

\bibitem[Albrecher et~al.]{albrecher2017optimal}
Albrecher H., Azcue P and Muler N. (2017).
\newblock Optimal dividend strategies for two collaborating insurance companies.
\newblock{\em Advances in Applied Probability}, 49, No.2, 515--548.

\bibitem[Albrecher et~al.]{albrecher2020optimal}
Albrecher, H., Azcue, P., and Muler, N. (2020).
\newblock Optimal ratcheting of dividends in insurance.
\newblock {\em SIAM Journal on Control and Optimization}, 58(4):1822--1845.



\bibitem[Albrecher et~al.]{albrecher2022optimal}
Albrecher, H., Azcue, P., and Muler, N. (2022).
\newblock Optimal ratcheting of dividends in a Brownian risk model.
\newblock {\em SIAM Journal on Financial Mathematics}, 13(3):657--701.

\bibitem[Albrecher et~al.]{albrecher2023optimal}
Albrecher, H., Azcue, P. and Muler, N. (2023).
\newblock Optimal dividends under a drawdown constraint and a curious square-root rule.
\newblock {\em Springer Finance and Stochastics}, 27(2):341--400.

\bibitem[Albrecher et~al.]{albrecher2018dividends}
Albrecher H., Nicole B. and Martin B.(2018).
\newblock Dividends: From refracting to ratcheting.
\newblock {\em Insurance Math. Econom}, 83, 47--58.


\bibitem[Aleksandorov]{aleksandorov1939almost}
Aleksandorov, AD. (1939).
\newblock Almost everywhere existence of the second differential of a convex function and some properties of convex functions.
\newblock {\em Leningrad Univ. Ann.}, 37:3--35.

\bibitem[Angoshitari et~al.]{angoshtari2019optimal}
Angoshtari B., Bayraktar E., and Young V.R. (2019).
\newblock Optimal dividend distribution under drawdown and ratcheting
constraints on dividend rates.
\newblock {\em SIAM Journal on Financial Mathematics}, 10(2):547--577.

\bibitem[Angoshitari et~al.]{angoshtari2022optimal}
Angoshtari, B., Bayraktar, E. and Young, V.R. (2022). 
\newblock Optimal investment and consumption under a habit-formation constraint. 
\newblock{\em SIAM Journal on Financial Mathematics}, 13(1):321--352.

\bibitem[Arun]{arun2012merton}
T~Arun. (2012).
\newblock The merton problem with a drawdown constraint on consumption.
\newblock {\em arXiv preprint arXiv:1210.5205}.

\bibitem[Asmussen and Taksar]{asmussen1997controlled}
 Asmussen, S. and   Taksar, M.. (1997).
\newblock Controlled diffusion models for optimal dividend pay-out.
\newblock {\em Insurance: Mathematics and Economics}, 20(1):1--15.

\bibitem[Avanzi]{avanzi2009strategies}
Avanzi, B. (2009).
\newblock Strategies for dividend distribution: A review.
\newblock {\em North American Actuarial Journal}, 13(2):217--251.

\bibitem[Azcue and Muler]{azcue2005optimal}
Azcue, P. and Muler, N. (2005).
\newblock Optimal reinsurance and dividend distribution policies in the
Cram{\'e}r-Lundberg model.
\newblock {\em Mathematical Finance}, 15(2):261--308.

\bibitem[Azcue et~al.]{azcue2014stochastic}
Azcue, P. and Muler, N. (2014).
\newblock Stochastic optimization in insurance: a dynamic programming approach.
\newblock {\em Springer}.

\bibitem[Belhaj]{belhaj2010optimal}
  Belhaj, M. (2010).
\newblock Optimal dividend payments when cash reserves follow a jump-diffusion
process.
\newblock {\em Mathematical Finance: An International Journal of Mathematics,
Statistics and Financial Economics}, 20(2):313--325.

\bibitem[Chen et al.]{chen2015on} 
Chen, X., Landriault, D., Li, B. and Li, D. (2015). 
\newblock On minimizing drawdown risks of lifetime investments. 
\newblock {\em Insurance: Mathematics and Economics}, 65:46--54.

\bibitem[Dai et al.]{dai2010continuous} 
Dai, M., Xu, Z.Q. and Zhou, X.Y. (2010). 
\newblock Continuous-time Markowitz's model with transaction costs.
\newblock {\em SIAM J. Financial Math.}, 1(1):96--125.


\bibitem[Dai and Yi]{dai2009finite}
Dai, M., and Yi, F.H. (2009).
\newblock Finite-horizon optimal investment with transaction costs: A parabolic double obstacle problem.
\newblock {\em J. Differential Equations}, 246:1445--1469.


\bibitem[De~Finetti]{de1957impostazione}
De~Finetti, B. (1957).
\newblock Su un'impostazione alternativa della teoria collettiva del rischio.
\newblock In {\em Transactions of the XVth international congress of
Actuaries}, volume~2, pages 433--443. New York.

\bibitem[Deng et al.]{deng2022optimal}
Deng, S., Li, X., Pham, H., Yu, X. (2022). 
\newblock Optimal consumption with reference to past spending maximum.
\newblock {\em Finance Stoch}, 26, 217--266.
 
\bibitem[Dybvig]{Dybvig1995dusenberry}
 Dybvig, P.H. (1995).
\newblock Dusenberry's ratcheting of consumption: optimal dynamic consumption and investment given intolerance for any decline in standard of living.
\newblock {\em The Review of Economic Studies}, 62, No.2, 287--313.

\bibitem[Elias]{elias2011}
 Elias, Z. (2011).
``Chapter 4. Function Limits and Continuity''. 
{\em Mathematical Analysis}, Volume I. p. 223.

\bibitem[Elie and Touzi]{romuald2008optimal} 
Elie, R. and Touzi, N. (2008). 
\newblock Optimal lifetime consumption and investment under a drawdown constraint. 
\newblock {\em Finance and Stochastics}, 12, 299--330.


\bibitem[Gerber]{gerber1969entscheidungskriterien}
Gerber, H.~U. (1969).
\newblock {\em Entscheidungskriterien f{\"u}r den zusammengesetzten
Poisson-Prozess}.
\newblock PhD thesis, ETH Zurich.


\bibitem[Gerber and Shiu]{gerber2004optimal}
  Gerber, H.U. and   Shiu, E.S.W. (2004).
\newblock Optimal dividends: analysis with brownian motion.
\newblock {\em North American Actuarial Journal}, 8(1):1--20.

\bibitem[Gerber and Shiu]{gerber2006optimal}
  Gerber, H.U. and   Shiu, E.S.W. (2006).
\newblock On optimal dividend strategies in the compound poisson model.
\newblock {\em North American Actuarial Journal}, 10(2):76--93.

\bibitem[Gu et~al.]{gu2017optimal}
Gu, J. W., Steffensen, M. and Zheng, H. (2018).
\newblock Optimal dividend strategies of two collaborating businesses in the diffusion approximation model.
\newblock{\em Mathematics of Operations Research}, 43, No.2, 377--398.

\bibitem[Guan et al.]{guan2023optimal2}
Guan CH, Fan JC, Xu ZQ. (2023).
\newblock Optimal dividend payout with path-dependent drawdown constraint.
\newblock  Preprint. Available at \url{https://arxiv.org/abs/2312.01668}



\bibitem[Guan and Xu]{guan2024optimal1}
Guan CH and Xu ZQ. (2024).
\newblock Optimal ratcheting of dividend payout under brownian motion surplus.
\newblock{\em SIAM Journal on Control and Optimization},  62, 2590--2620.


\bibitem[Jensen]{jensen1988the}
Jensen, R. (1988).
\newblock The maximum principle for viscosity solutions of second order fully nonlinear partial differential equation.
\newblock {\em Arch. Rat. Mech. Anal.}, 101, 1·27.

\bibitem[Jeon et al.]{jeon2018portfolio}
Junkee J., Hyeng~K.K., and Yong~H.S.. (2018).
\newblock Portfolio selection with consumption ratcheting.
\newblock {\em Journal of Economic Dynamics and Control}, 92:153--182.


\bibitem[Keppo et al.]{KRS21}
Jussi Keppo, A~Max Reppen, and H~Mete Soner. (2021).
\newblock Discrete dividend payments in continuous time. 
\newblock {\em Mathematics of Operations Research}, 46, 3, 895-911.

\bibitem[Reppen et al.]{reppen2020optimal}
  Reppen, A.M.,  Rochet, J.C., and   Soner, H.M. (2020).
\newblock Optimal dividend policies with random profitability.
\newblock {\em Mathematical Finance}, 30(1):228--259.

\bibitem[Roche]{herve2006optimal}
Roche, H. (2006).
\newblock Optimal consumption and investment strategies under wealth ratcheting.
\newblock Preprint. Available at \url{https://gente.itam.mx/hroche/Research/MDCRESFinal.pdf}

\bibitem[Schmidli]{schmidli2007stochastic}
 Schmidli, H. (2007).
\newblock Stochastic control in insurance.
\newblock {\em Springer Science \& Business Media}.

\bibitem[Yong and Zhou]{YZ99}
Yong, J., and  Zhou, X.Y. (1999).
{\em Stochastic controls: Hamiltonian systems and HJB equations.}
Applications of Mathematics (New York) 43,
Springer-Verlag, New York.

\end{thebibliography}
 
\end{document}